\documentclass[11pt, a4paper, oneside]{article}
\usepackage[body={16.5cm, 21.0cm},left=1.5cm,right=1.5cm]{geometry}

\usepackage{natbib}
\usepackage{amsmath, amssymb, amsthm, amsfonts}
\usepackage{hyperref}
\usepackage{float}
\allowdisplaybreaks
\usepackage{graphicx,xcolor}
\usepackage{dsfont} % to use one for denoting indicator function
\usepackage{nicefrac}
\usepackage{multirow} % for tables
\usepackage{comment} % for commenting a large chunk of text
\usepackage{accents}
\usepackage{pifont}

\usepackage{utopia}
\usepackage{setspace}
\usepackage{fancyhdr}
\numberwithin{equation}{section}

\definecolor{mygreen}{rgb}{0.30,0.44,0.07}
\definecolor{orange}{rgb}{1.00,0.65,0.00}

\theoremstyle{plain}

\newtheorem{thm}{Theorem}
\newtheorem{lem}[thm]{Lemma}

\newtheorem{prop}[thm]{Proposition}

\theoremstyle{definition}
\newtheorem{rem}{Remark}

\newenvironment{pfofThm}{\noindent{\bf Proof of Theorem}}{\hfill $\square$ \\}
\newenvironment{pfofLem}{\noindent{\bf Proof of Lemma}}{\hfill $\square$ \\}

\newcommand{\id}{\ensuremath{\displaystyle{\mathop {=} ^d}}}
\newcommand{\as}{\ensuremath{\displaystyle{\mathop {=} ^{a.s.}}}}
\newcommand{\approxid}{\ensuremath{\displaystyle{\mathop {\approx} ^d}}}

\newcommand{\field}[1]{\mathbb{#1}}
\newcommand{\real}{\ensuremath{{\field{R}}}}

\newcommand{\sumab}[2]{\ensuremath{\sum\limits_{#1}^{#2}}}
\newcommand{\intab}[2]{\ensuremath{\int_{#1}^{#2}}}

\newcommand{\intinf}[1]{\ensuremath{\int_{#1}^{\infty}}}

\newcommand{\arrowf}[1]{\ensuremath{\displaystyle {\mathop {\longrightarrow}_{#1 \rightarrow \infty}\,}}}
\newcommand{\limit}[1]{\ensuremath{\displaystyle {\lim_{#1 \rightarrow{\infty}}}}}
\newcommand{\suprem}[1]{\ensuremath{\displaystyle {\sup_{#1}}}}

\newcommand{\conv}[1]{\ensuremath{\, \displaystyle {\mathop{\longrightarrow}^{#1}}}\, }

\newcommand{\one}{\mathds{1}}

\definecolor{mygreen}{rgb}{0.30,0.44,0.07}
\definecolor{myorange}{rgb}{1.00,0.65,0.00}
\title{A hybrid-Hill estimator enabled by heavy-tailed block maxima}

\author{Cl\'{a}udia Neves \\  
 {\small King's College London}\\ 
  \and {Chang Xu}\\ {\small King's College London}
}

\date{}

\begin{document}

\maketitle

%\strut

%==============MAIN TEXT==============================
%======================================================
\begin{abstract}
When analysing extreme values, two alternative statistical approaches have historically been held in contention: the block maxima method (or annual maxima method, spurred by hydrological applications) and the peaks-over-threshold.
Clamoured amongst statisticians as wasteful of potentially informative data, the block maxima method gradually fell into disfavour whilst peaks-over-threshold-based methodologies climbed to the centre stage of extreme value statistics.
This paper devises a hybrid method which reconciles these two hitherto disconnected approaches. Appealing in its simplicity, our main result introduces a new universality class of extreme value distributions that discards the customary requirement of a sufficiently large block size for the plausible block maxima-fit to an extreme value distribution. Natural extensions to dependent and/or non-stationary settings are mapped out. We advocate that inference should be drawn solely on larger block maxima, from which practice the mainstream peaks-over-threshold methodology coalesces: the asymptotic properties of the hybrid-Hill estimator herald more than its efficiency, but rather that a fully-fledged unified semi-parametric stream of statistics for extreme values is viable. A reduced-bias off-shoot of the hybrid-Hill estimator provably outclasses the incumbent maximum likelihood estimation that relies on a numerical fit to the entire sample of block maxima.
\end{abstract}


%=======================================
\section{Introduction}

Despite there being various approaches by which extreme values can be statistically analysed, these have historically been categorised into two main strands that eventually evolved in separate ways:
methods for maxima over fixed-length intervals and methods for exceedances (or peaks) over a high threshold. The former relates the oldest group of extremal models, arising from the seminal work of Gumbel (1958) as block maxima (BM) models to be fitted to the largest observation collected from each of the large samples (blocks) of identically distributed observations. The peaks-over-threshold's (POT's) strand of development seems to have deployed from \citet{TodorovicZelenhasic1970}. The impacts of adopting either a BM method or the POT approach are still being thrashed out. Few studies have looked at the issue and many that have are aimed at a judicious selection of one approach over the other  -- a challenge more often than not won by the POT; see \cite{BucherZhou2021} -- rather than addressing them with a reconciliatory stance. A common problem is that studies fail to ensure that POT is not applying a BM method as well. While the BM method lends itself well to lower-frequency, long-span data, the POT approach is preferable for high-frequency sequences such as daily or hourly observations. This is typically the narrative to discerning  between approaches for practical applications. The nuance though is reflected on daily maxima, just to give an example, where there is a level of aggregation in the data and only maximum daily temperatures have been observed. Another example is when the interest is in hourly maximum windspeed or wind gusts. To the best of our knowledge, the only closely related approach for statistical inference in this direction is the work by \cite{Wager2014}. But, so too, their relatively elaborate subsampling procedure entails that the resulting extreme values estimation must concede to the provision of a block size $m(n)$ going to infinity with the sample size $n$. In our proposed hybrid approach, the large enough block size $m$ for the extreme value theorem to hold within each block is a catalyst instead of a requirement. This is found to translate into significant gains to the estimation of extreme value characteristics, especially that of a positive extreme value index. Since our estimators do not depend on the choice of the block size, and indeed are robust to changes in $m$, our findings add significantly to the current body of knowledge.

This paper addresses the estimation of the positive extreme value index $\gamma >0$, regarded as a measure of tail-heaviness, in a principled way that conflates two hitherto separate statistical inference tenets in extreme value statistics, both BM and POT methodologies. We begin with the classical BM method which severs the dataset into non-overlapping blocks and models only the maximum value from each block to an asymptotically prescribed extremal value distribution. Since it discards sub-maximal observations, the BM framework is robust to measurement errors and short-term irregularities, making it well suited to datasets with lower temporal resolution or long historical spans. But while shorter blocks do not fully mitigate seasonal cycles and impair within-block convergence to the limiting extreme value  distribution, larger blocks reduce the effective sample size and lead to a wastage of information. A similar criticism applies to de-clustering methods. According to \citet[][pp.100,178]{Coles2001}, results can be sensitive to the arbitrary choices made in cluster determination, especially running at a loss of information as a consequence of discarding all data except the cluster maxima, with missed opportunities for modelling within-cluster behaviour since there may be short-range dependence within blocks but not between blocks \citep{Katz2002}.

Closely related to Hill's estimation, it can be deduced from a possible extension to the case $r=0$ of Theorem 2.2 in \citet[][albeit for $\gamma < 1/2$, as $m\rightarrow \infty$]{FerreiradeHaan2015} that the Hill's statistic is a valid estimator of $\gamma >0$. But if employed using the entire sample of BM, as current practices advise for averting wastefulness of data, the asymptotic variance of this estimator results very large with no closed form expression as a function of $\gamma$. Figure \ref{Fig:VaryingBlockSize} offers an early sight of what our proposed hybrid estimation can achieve when using a sufficiently large BM-underpinned random threshold that does not depend on the block size $m$. In Figure \ref{Fig:VaryingBlockSize}, the Hill's estimates are computed using the top $25\%$ of all generated $k$ block maxima. The $k$ blocks, data slices of varying size $m= 1, 2, \ldots, 100$, tally to a sample of size $n=10,000$. The maximum likelihood estimation was conducted in the usual way, based on the numerical fit of the limiting GEV to the whole sample of BM. These preliminary simulation results evidence that it is preferable to have fewer observations within each block and many more BM for minimising the MLE's mean squared error.

%------- Figure emphasising block length-------------
\begin{figure}
\begin{center}
\includegraphics[scale=0.25]{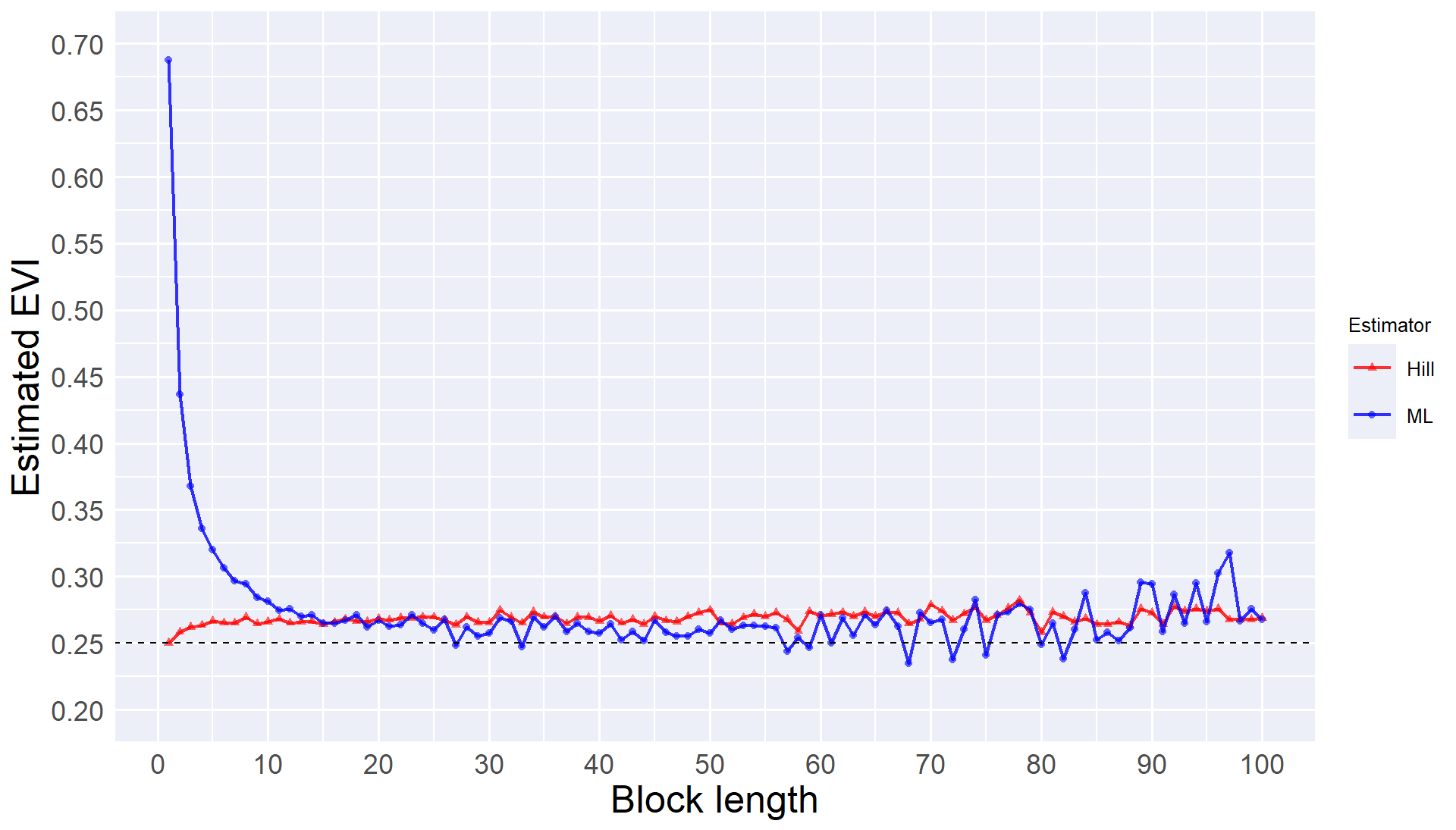} 
\includegraphics[scale=0.25]{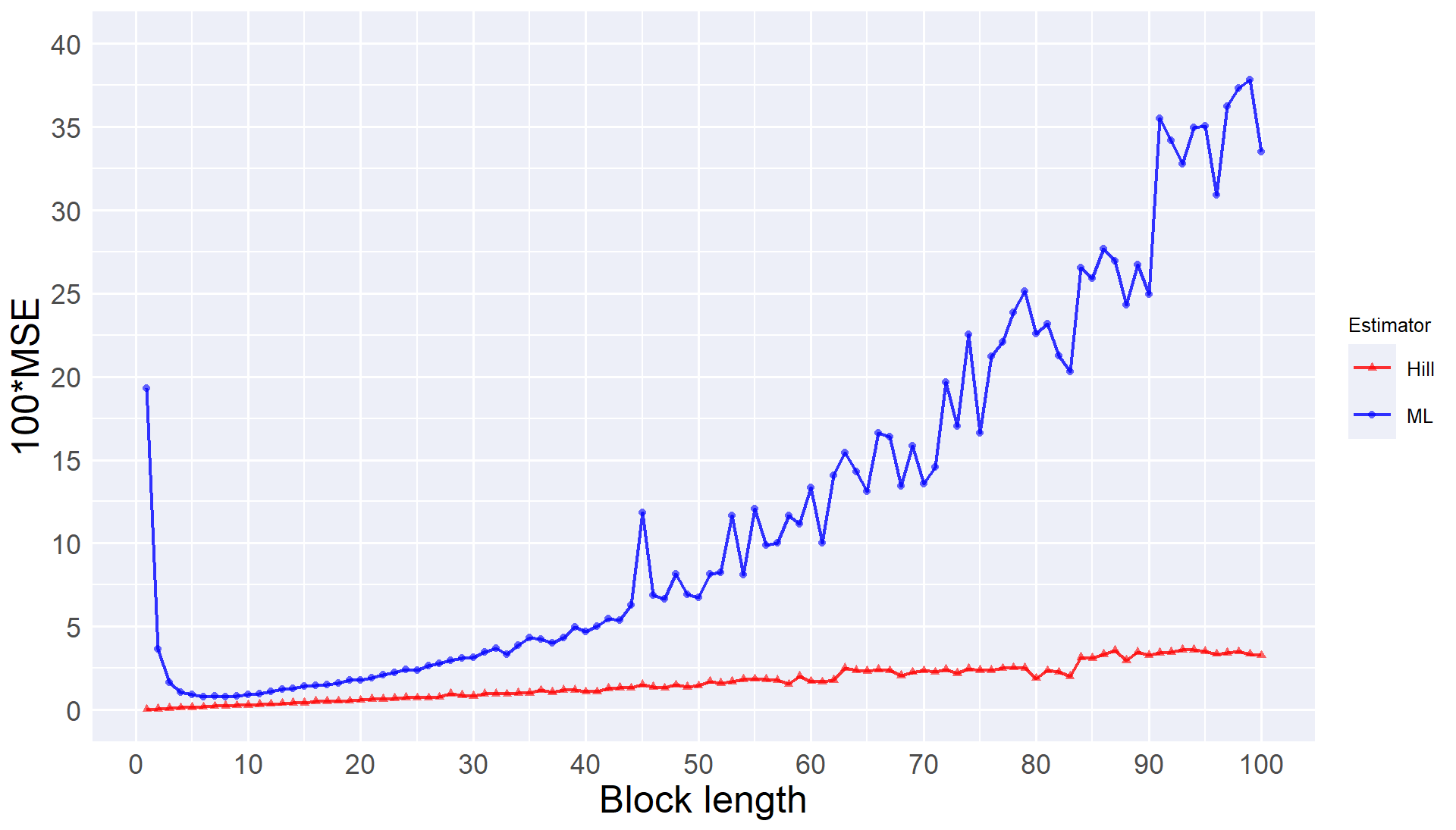}
\end{center}
\caption{Average estimates of both hybrid-Hill and GEV-maximum likelihood estimates and their respective empirical mean squared errors (MSE), plotted against the block size $m=1, \ldots, 100$. Simulations consist of $1000$ replicates of a sample of $n=10,000$ i.i.d. observations taken from a Pareto distribution with extreme value index $\gamma =1/4$.}
\label{Fig:VaryingBlockSize}
\end{figure}

On the other hand, the inadequacy of small block sizes for the extreme value theorem to hold is offset by the added flexibility for setting a larger threshold in the hybrid-Hill estimation, with augmented capacity for accurate inference that actually takes places in the tail region. Eventually, for smaller block sizes nearing a single data-point, the BM method seamlessly defaults to the mainstream POT approach ascribed to $m=1$. The assumptions involved, laid out just next in Section \ref{SSec:Struct}, do allow every $m\geq 1$. By construction, the hybrid method discards those lower block maxima that  might be too small to be considered representative of extremal behaviour and would otherwise hinder a a proper GEV parametric fit. One such applied situation relates the contribution of very dry months to assessing extreme monthly rainfall, whose lowest maxima values would scupper every attempt to find a plausible GEV-fit to the whole sample of BM. Not for nothing, avoiding this pitfall is deemed a necessity, formalised in Condition 3.2 of \citet{BucherSegers18}, albeit in a time-series context.\\

%\subsection{The Hybrid-Hill estimator: definition and framework}

To fix ideas, let us begin with a sequence $X_1,X_2, \ldots , X_n, \dots$ of independently and identically distributed (i.i.d.) random variables with common distribution function (d.f.) $F$ belonging to the max-domain of attraction of an extreme value distribution $G$ in the Fr\'echet family of distributions. That is, we assume that there exist constants $a_n>0$ such that
\begin{equation}\label{eq:DoA}
\limit{n} P\bigl\{ a_n^{-1}\max(X_1, \ldots, X_n) \leq x \bigr\} =  \limit{n} F^n\bigl(a_n x\bigr) = G(x),
\end{equation}
for every continuity point of  $G$. Then, $G(x)=G_{\gamma}\bigl((x-1)/\gamma\bigr)$ and $G_\gamma$ is the Generalised Extreme Value (GEV) distribution function given by
\begin{equation*}
    G_\gamma(x)=\exp\bigl\{-(1+\gamma x)^{-\frac{1}{\gamma}}\bigr\},
\end{equation*}
for all $x$ such that $1+\gamma x>0$, whereby we say that $F$ belongs to the domain of attraction of the Fr\'echet distribution with distribution function $\exp\{-x^{-1/\gamma}\}$, $x\geq 0$, for some extreme value index $\gamma>0$. Let us note that if there is a limiting distribution to the i.i.d. partial sequence of linearly normalised maxima then that distribution in the limit must be max-stable. In particular, if ascribed to $\gamma >0$, the extreme value types theorem  \citep[due in various degrees of
generality to][]{FisherTippett28,Gnedenko43,deHaan1970} statement in \eqref{eq:DoA} is tantamount to the following limit relation running instead on the set of real values: for all $x \geq 0$,
\begin{equation}\label{eq:BasicDoA}
	\limit{t} \frac{1}{-t \log F \bigl(a(t)x \bigr)} = x^{1/\gamma}.
\end{equation}
% \citep[cf.][p.177]{deHaanFerreira2006}. 

We now deviate from the mainstay expansion $t (1- F(a(t)x))$ applied to the left hand-side of \eqref{eq:BasicDoA} for the establishing the dual extreme value types theorem that stipulates the Generalised Pareto distribution as the limit of POT observables \citep[cf.][Section~1.1.2 see p.177, e.g.]{deHaanFerreira2006}, to focus instead on the limit of a suitably normalised $1-F^{t}(a(t)x)$, for $x \rightarrow \infty$.
This paves the way for a characterisation of max-domains that is very similar to that for pot-domains of attraction towards the prescribed Pareto limit, eventually unifying the two. While the domain of attraction condition \eqref{eq:BasicDoA} appertains to the unknown distribution $F$ generating the raw data, our identified links between POT and BM conditions will arise from analogous considerations around the tail distribution function of block maxima $1-F^t$. Crucially, the proposed Hill estimator of $\gamma>0$ stems from this recasting for obtaining the universality class that encloses both BM and POT domains of attraction, analysed one at a time in \citet{BucherZhou2021}. This is the key insight to the aimed unifying statistical methodology for extreme values: inference conditional on larger block maxima, yet still predicated on the negative-log transform of the initial distribution function $F$ as in \eqref{eq:BasicDoA}.

Here and throughout, the block maxima (BM) are represented by the sequence of random variables $\{M_{i}^{(m)}\}_{i=1}^k$ defined, for $m= 1, 2, \ldots$ and $i= 1, \ldots, k$ $(k \in \mathbb{N})$ as
\begin{equation}\label{eq:BMdef}
    M_{i}^{(m)} :=\mathop{\max}\limits_{(i-1)m<j<im}X_{j}\,,
\end{equation}
and constitute an i.i.d. sample taken from the distribution $F^m $.
It is worth our while to highlight at this point that the distribution $F^{m}$ is in the same max-domain of attraction of the initial distribution function $F$: the two distribution functions are regularly varying in the tail with the index $-1/\gamma <0$ and $1-F^m$ has second-order parameter governing the rate of that convergence matching that of $1-F$ (see Appendix \ref{App:DoAFm}); their extreme value limiting distributions coincide as in the embodiment of the max-stability property exclusive to the only three possible extreme value distributions, which has earned them formidable uses in vastly many applications  \citep{Beirlantetal2004,Resnick2007,ReissThomas2008}. % see Proposition 2.12 of Resnick's Extreme values, regular variation and point process; Exercise 3.17 of Resnick's Heavy-Tail Phenomena which directs us to the paper by Geluk, de Haan, Resnick and Starica (1997) Second-order regular variation, convolution and the central limit theorem.

Through the blocking mechanism \eqref{eq:BMdef} for obtaining a sample made up of $k$ independent BM with the same distribution function $F^m$, our proposed hybrid-Hill estimator (H2 estimator) employs the $k_0$ upper order statistics of the partial BM sequence. For every block-length $m \geq 1$, let $M^{(m)}_{1:k} \leq M^{(m)}_{2:k} \leq \ldots \leq M^{(m)}_{k:k}$ be the order statistics of $\{M_{i}^{(m)}\}_{i=1}^k$. The H2 estimator of $\gamma >0$ is defined as
\begin{equation}\label{eq:HibHill}
    \hat{\gamma}^{\textrm{H2}} \equiv  \hat{\gamma}^{\textrm{H2}}(k_0,k) := \frac{1}{k_0} \sumab{i=0}{k_0-1} \log M^{(m)}_{k-i:k} -\log M^{(m)}_{k-k_0:k}\, .
\end{equation}

%-------------------------------------------
\subsection{Structure of the paper}\label{SSec:Struct}

We shall assume throughout that $\{X_{n}\}_{n \geq 1}$ is an independent (possibly weakly dependent) sequence of random variables whose common distribution function $F$ satisfies the primary extreme value condition \eqref{eq:BasicDoA}. Then, by continuity of the limiting Fr\'echet distribution function  and monotonicity of the left hand-side of \eqref{eq:DoA}, the convergence in \eqref{eq:BasicDoA} is uniform, whereby
\begin{equation}\label{eq:UnifConvDoA}
	\limit{t} \, \suprem{x \in \real} \,\Bigl| \frac{1}{- t \log F(a(t)x)} - x^{1/\gamma}\Bigr| = 0.
\end{equation}
Equivalent to \eqref{eq:UnifConvDoA} is the next condition formulated in terms of the quantile-type function
\begin{equation*}
	V:= \bigl( 1/ (-\log F) \bigr)^{\leftarrow}= F^{\leftarrow} (e^{-1/t}),
\end{equation*}
where $^{\leftarrow}$ denotes generalised inverse.

\emph{Condition A1.} Assume that: %the equivalent conditions to \eqref{eq:UnifConvDoA} by taking $a(t) = V(t)$:
(i) $ \lim_{t\rightarrow \infty}
	\log V(tx) - \log V(t) = \gamma \log x$, for all $x>0$; (ii)
for every $\varepsilon,\delta>0$ there is a $t_0= t_0(\varepsilon,\delta)$, such that for $t\geq t_0$ and $x\geq 1$,
\begin{equation*}
    \Bigl|\frac{V(tx)}{V(t)}-x^{{\gamma}}\Bigl|\leq \varepsilon x^{{\gamma}+\delta}.
\end{equation*}
We then say that $V$ is regularly varying with index $\gamma$, i.e., it satisfies 
\begin{equation*}%\label{eq:RVofV}
    \limit{t} \frac{V(tx)}{V(t)}=x^{\gamma},
\end{equation*}
for all $x>0$, and use the notation $V \in RV_{\gamma}$.

Together with the max-stability property of the limiting extreme value distribution $G_\gamma$, the uniform relation \eqref{eq:UnifConvDoA} brings to light the coincidence between the classes of distributions $F$ and $F^{t}$ induced by each of their specific max-domains of attraction conditions. What is involved in unifying the two classes stems from the asymptotically equal rescaling by $V(t) \sim a(t)$, as $t \rightarrow \infty$ \citep[cf.][Proposition 2.12, p.108]{Resnick1987}.

\textbf{This paper is organised as follows.} Firstly, we introduce a cursory semi-parametric estimator of $\gamma>0$, whose asymptotic behaviour encompassed \textbf{in Section \ref{Sec:CoEstim}} can be read off from Condition A1 in self-evident way. Since the goal is to widen the focus to the tail distribution function underlying the BM picked up from blocks of size $m$, we devise a new set of extreme value conditions which will unveil the Pareto-tail as the proper limit for the regularly varying $1-F^m$ upon inner normalisation by $V(m(t-1/2))$.

\emph{Condition B.} With $V$ satisfying Condition A1, assume the following holds for every $m \geq 1$: 
\begin{equation}\label{eq:RVofFm}
    \lim_{t\rightarrow \infty} \frac{1-F^m\bigl( x V(m(t-1/2)) \bigr) }{ 1 - F^m \bigl( V(m(t-1/2)) \bigr)}= x^{-1/\gamma},
    \end{equation}
    for all $x>0$.
Since $V$ is nondecreasing, there are accompanying uniform bounds in the sense of Proposition B.1.10 of \citet{deHaanFerreira2006}: for each $\varepsilon,\delta>0$ there is $t_0= t_0(\varepsilon,\delta)$ such that for $t$, $tx\geq t_0$,
\begin{equation*}
    \biggl|\frac{1-F^m\bigl( x V(m(t-1/2)) \bigr) }{ 1 - F^m \bigl( V(m(t-1/2)) \bigr)}-x^{-\frac{1}{\gamma}}\biggl|\leq \varepsilon \max(x^{-\frac{1}{\gamma}-\delta}, \,x^{-\frac{1}{\gamma}+\delta}).
\end{equation*}

Proposition \ref{Prop:DoACondMax} in Appendix \ref{App:CondB1} ascertains that Condition B does provide a universal and complete characterisation of max-domains of attraction, irrespective of $m \geq 1$ being set fixed or assumed going to infinity with $n$. 

In order to make matters simple and furnish an extreme-value condition that is fit for our purposes, we proceed to the estimation on the basis of Condition B, since the normalisation with $V(m(t-1/2))$ makes it possible for inference to stretch back into the tail region so as to harness intermediate BM values in addition to the extreme. This renders a semi-parametric approach at the intersection of BM and POT methodology in a fruitful and impactful way for applications. \textbf{In Section \ref{Sec:HybridHill}}, the hybrid-Hill estimator presented in \eqref{eq:HibHill} is engendered on the empirical counterpart of Condition B. To heap on the evidence that BM and POT approaches are effectively unified, we note that such a development conforms to a wealth of already published work around the asymptotic properties of the original Hill's estimator \citep{Hill1975} and more widely to large sample properties of semi-parametric estimators specialised in (log-)exceedances above a high threshold.

 The speed of convergence in Condition A1, and thus that of \eqref{eq:RVofFm}, is not the same for all parent distributions, but it can be made explicit through the second-order refinement in the next condition. Condition A2 will prove instrumental to capturing the approximation bias, a deterministic pre-asymptotic bias, that such deviations from the limiting Pareto-tail instil in the H2 estimator. Of relevance in this regard is Lemma \ref{Lem:2ndOrdLogShift} in Appendix \ref{App:CondB2}, which demonstrates in particular that the convergence in Condition B can be much slower than that with Condition A1. 
 
\emph{Condition A2}: Suppose there exists a function $A$ ultimately of constant sign and $\lim_{t\to \infty}A(t)=0$, such that
\begin{equation}\label{eq:2ndRVofV}
     \limit{t}\frac{\frac{V(tx)}{V(t)}-x^{\gamma}}{A(t)}=x^{{\gamma}}\frac{x^{\rho}-1}{\rho},
    \end{equation}
for all $x>0$. This convergence is necessarily locally uniform. Moreover, $|A| \in RV_{\rho}$, $\rho \leq 0$ and since $A$ converges to zero then \eqref{eq:2ndRVofV} is equivalent to 
\begin{equation}\label{eq:2ndRVLogV}
    \limit{t}\frac{\log V(tx)-\log V(t)-\gamma \log x}{A(t)}=\frac{x^{\rho}-1}{\rho},
\end{equation}
for all $x>0$.
Consequently, Theorem B2.18 of \citet{deHaanFerreira2006} applied against the background of Condition A2 establishes, for a possibly different function $A_{0}$, with $A_{0}(t)\sim A(t)$, $t\rightarrow\infty$, and for each $\varepsilon,\delta>0$ that there exists a $t_0$ such that for $t\geq t_0$, $x\geq 1$, the following inequality holds:
\begin{equation}\label{inequality}
    \Bigl|\frac{\log V(tx)-\log V(t)-\gamma \log x}{A_0(t)}-\frac{x^{\rho}-1}{\rho}\Bigr|\leq\varepsilon x^{\rho+\delta}.
\end{equation}
We then say that the quantile-type function $V$ is second order regularly varying at infinity with second order parameter $\rho \leq 0$ governing the rate of convergence in Conditions A1(i)-(ii) and use the notation $V \in 2RV(\gamma, \rho)$.

The remainder of the paper is outlined as follows. \textbf{In Section \ref{Sec:OptimalFraction}} we gather useful results concerning the optimal selection of the top fraction of BM employed in the H2 estimation of $\gamma>0$.  Some of these considerations are carried forward to \textbf{Section \ref{Sec:RBH2}}, where a bias reduction procedure is devised that pieces together the cursory estimator from Section \ref{Sec:CoEstim} and the H2 estimator's asymptotic results expounded in Section \ref{Sec:HybridHill}. To this end, we shall adopt a refinement of the class of regularly varying functions $V$ by allowing a uniform convergence condition that combines \eqref{eq:2ndRVLogV} with Condition B, and through which mapping we get a still much rich class that corrals distributions $F^m$ indexed by two parameters $(\gamma, \tilde{\rho}) \in \real^+ \times [-1, \, 0]$.
Section \ref{Sec:RBH2} is interspersed with simulation results aimed at elucidating how the H2 estimator succeeds in its innovative use with BM. Used in tandem with the cursory estimator, the H2 estimator gets subtracted of its approximation (second order) bias and asserts itself as a serious competitor in a level playing field that visibly merges BM and POT approaches. 
Those more technical proofs are deferred to \textbf{Section \ref{Sec:Proofs}}.

A comprehensive simulation study is available as \textbf{supplementary material} to this paper.

\subsection{Notation}

In this paper, we use $F\in D(G_{\gamma})$ to denote the max-domain of attraction of $F$, i.e., the set of all distribution functions $F$ for which the BM have the same type of asymptotic distribution $G_{\gamma}$, $\gamma \in \real$. We have used and will continue to use the condensed notations $\overline{F}:= 1-{F}$,  $\overline{F^{m}}:= 1-{F}^{m}$ and $g(t):= t-1/2$. 
Regular variation with index $\gamma$ is $RV_\gamma$ and its second order refinement is $2RV(\gamma, \rho)$. For $u \in \real$, let as usual $[u]$ represent the smallest integer greater than or equal to $u$. %EVI stands for extreme value index $\gamma \in \real$.

\subsection{N.N.I.I.D. random variables or I.N.N.I.D. random variables}\label{Sec:ennes}

The standard practice of identifying clusters of high-level exceedances with the intention of concentrating on cluster maxima for data analysis is a form of data aggregation that not only is permissible but also seamlessly covered by the hybrid semi-parametric framework. Each block containing at least one threshold exceedance, hence treated as a cluster, is one situation under the general umbrella of not necessarily independent identically distributed (N.N.I.I.D.) random variables which exerts further reduction of the block size through the approximate equality
\begin{equation*}
	P(M_i^{(m)} > x) \approx 1- F^{m\theta}(x),
\end{equation*}
where $0 < \theta \leq 1$ is the so-called extremal index. Since Condition B-\eqref{eq:RVofFm} can hold for every $i =1, 2, \ldots, k$, universality class it induces makes our results particularly suited for tackling extreme value estimation problems where the independence assumption must be dropped.
Furthermore, Condition B remains valid under mild mixing conditions such as Leadbetter's condition $D$.

The new setting of regular variation for characterising maxima taken from a heavy-tailed distribution, encapsulated in Conditions A1 and B and including their respective second order conditions (details in Appendices \ref{App:CondB2} and \ref{App:DoAFm}) inherits very many of the existing asymptotic theory for independent not necessarily identically distributed (I.N.N.I.D.) random variables in a plentiful manner. It is a straightforward exercise to assemble the universality Condition B into the theorem of convergence of heteroscedastic extremes in \citet{deHaan2015}, from which result a unifying stream of statistical methods for a trend in extremes will emanate \citep[see e.g.][]{TailTrend15,SpaceTimeTrend22}. We refer to Appendix \ref{App:CondB1}.

The now proposed hybrid-estimation methodology encompasses all the preceding settings.
The fact that the distribution function $F^{\theta m}$ may still be a plausible approximation to $\prod_{j=1}^{\theta m}F_j$, for every $\theta^{-1} \geq 1$, so long as the deviations of $F_j$ from one another do not place them in different max-domains of attraction, adds significance to the hybrid framework.
 Notably, due to the lack of asymptotic theory for non-stationary BM, current practice has been limited to the use of the GEV distribution as an emulator of the random process generating BM and non-stationarity is modelled by allowing its parameters to change with time and other covariates \citep[see e.g.][]{DupuisTrapin23}. The hybrid semi-parametric estimation fills in this gap in current knowledge, since as a result of fusing BM and POT approaches into a single  methodological strand, most of the extreme values modelling techniques that have to date carried the hallmarks of each approach have just become transferable and ready for uptake.

%============= Section First estimator ====================
\section{A cursory semi-parametric hybrid estimator}\label{Sec:CoEstim}

Let $M^{(m)}_{1:k} \leq M^{(m)}_{2:k} \leq \ldots \leq M^{(m)}_{k:k}$ denote the order statistics associated with the i.i.d. BM defined in \eqref{eq:BMdef}. A cursory non-parametric estimator for $\gamma >0$,  which will be the focus in this section, is given by
\begin{equation}\label{eq:EstNaive}
    \hat{\gamma}(\theta_k, k):= (\log 2)^{-1}\bigl( \log M^{(m)}_{k-[\theta_k/4]:k} - \log M^{(m)}_{k-[\theta_k/2]:k} \bigr),
\end{equation}
where $\theta_k \ll k$. Its large sample consistency is established in the following theorem.

\begin{thm}\label{Thm:ConsistencyNaive} Let $\{X_n\}_{n\geq 1}$ be a sequence of i.i.d. random variables drawn from a distribution $F$ such that Condition A1 holds. For any block length $m\geq 2$, the cursory estimator $\hat{\gamma}(\theta_k, k)$ defined in \eqref{eq:EstNaive} is a consistent estimator for $\gamma >0$ in the sense that, as $k \rightarrow \infty$,
\begin{equation*}
    \hat{\gamma}(\theta_k, k) \conv{P} \gamma, 
\end{equation*}
provided $\theta_k \in (0, k]$ is an upper intermediate sequence, i.e., such that $\theta_k \rightarrow \infty$ and $\theta_k/ k\rightarrow 0$, as $k \rightarrow \infty$.
\end{thm}

The tentative estimator $\hat{\gamma}(\theta_k, k)$  reflects a fairly simple construct that can be made much more general. The next result, whose proof is part of Section \ref{Sec:Proofs}, furnishes an anchor-step theorem towards this aim.

%---------- Lemma -------------
\begin{lem}\label{Lem:UnitFrechetOS}
Let $Z_{1:k} \leq Z_{2:k}\leq \ldots \leq Z_{k:k}$ be the order statistics of the sample $(Z_1, Z_2, \ldots, Z_k)$ of i.i.d. unit Fr\'echet random variables with distribution function $e^{-1/z}$, $z \geq 0$. Suppose $(\theta_k)$ is a sequence of positive integers such that $\theta_k \rightarrow\infty$ and $\theta_k=o(k)$, as $k \rightarrow \infty$. Then, the following hold:
\begin{itemize}
    \item[(i)] $\displaystyle{ \sqrt{\theta_k} \, \Bigl\{ \frac{\theta_k}{k}  \bigl(Z_{k - \theta_k :k} + 1/2 \bigr)-1 \Bigr\} \conv{d} \, N(0,1) }$, as $k \rightarrow \infty$.
     \item[(ii)] With $M^{(m)}_{1:k} \leq M^{(m)}_{2:k}\leq \ldots \leq M^{(m)}_{k:k} $, the order statistics associated with the BM defined in \eqref{eq:BMdef}, for every $m \in \field{N}$, we have the approximation in distribution:
     \begin{equation*}
         \Bigl\{  M^{(m)}_{k-i:k} \Bigr\}_{i=0}^{\theta_k} \mathop{\approx}^d \,  \Bigl\{ V\Bigl( m \exp\bigl\{1/Z_{i+1:k} \bigr\} -\frac{m}{2}\Bigr) \Bigr\}_{i=0}^{\theta_k}.
     \end{equation*}
\end{itemize}
\end{lem}

For the proof of Theorem \ref{Thm:ConsistencyNaive}, we will need a result akin to Lemma \ref{Lem:UnitFrechetOS}(i) but stated with greater generality. Lemma \ref{Lem:UnitFrechetStdProcess} fulfils this aim of strengthening Lemma \ref{Lem:UnitFrechetOS} so as to inherit more in the form of an invariance principle like those typically sought in empirical processes theory for ascertaining the proper rate of convergence to a Donsker-type limit in the uniform sense.
%---------- Lemma -------------
\begin{lem}\label{Lem:UnitFrechetStdProcess}
Let $Z_1, Z_2, \ldots, Z_k, \ldots $ be i.i.d. random variables from a unit Fr\'echet distribution with d.f. $\exp(-x^{-1})$, $ x\geq 0$. Suppose $\theta_k \in (0, k]$ is such that $\theta_k \rightarrow\infty$ and $\theta_k/ k\rightarrow 0$, as $k \rightarrow \infty$. Then, with a sufficiently rich probability space, there exists a sequence of Brownian motions $\bigl\{ W_k(s)\bigr\}_{s \geq 0}$ such that, for each $\varepsilon >0$,
\begin{equation}\label{UnitFrecheQuantileProcess}
    \suprem{0 < s \leq 1} s^{1/2 + \varepsilon} \, \biggl| \sqrt{\theta_k} \Bigl( \frac{\theta_k s}{k} \exp\bigl\{Z^{-1}_{[\theta_k s]+ 1:k}\bigr\} -1 \Bigr) - \frac{W_k(s)}{s}\biggr| = o_p(1).
\end{equation}
\end{lem}
%----------------------------------
\begin{proof}
    It suffices to employ Lemma 2.4.10 of \citet{deHaanFerreira2006} with $\gamma=1$ and $\theta_k$ in place of $k$ therein, whilst heeding the law of iterated logarithm in that, for $\varepsilon >0$,
    \begin{equation*}
        \lim_{\theta_k^{-1} \downarrow 0}\,  \suprem{0 < s \leq \theta_k^{-1}} \bigl| s^{-1/2 + \varepsilon} W(s)\bigr| = 0 \quad a.s.,
    \end{equation*}
    for ensuring tightness (i.e. that no probability escapes bounds) near the maximum $\exp\{Z^{-1}_{ 1:k}\}$.
\end{proof}
%----------------

With $s$ on a compact interval bounded away from zero, Lemma \ref{Lem:UnitFrechetStdProcess} sets the scene to establishing that the Gaussian process as the limit in \eqref{UnitFrecheQuantileProcess} remains unaltered in the space $D(0,T])$, $T^{-1} >1$, in the midst of replacing the basic process $\bigl\{\exp\{Z^{-1}_{[\theta_k s]+ 1:k}\}\bigr\}$ with the relevant quantile process $\{M^{(m)}_{k- [\theta_k s]:k}\}_{s \in (0,1]}$ describing intermediate BM and upwards to near extreme BM.\\

%------- Proof of consistency of naive estimator --------
\begin{pfofThm}~\ref{Thm:ConsistencyNaive}:
We begin by showing that $M^{(m)}_{k - [\theta_k/2]:k}/ V\Bigl( m \frac{2k}{\theta_k} - \frac{m}{2} \Bigr) \conv{P} 1$, as $k \rightarrow \infty$. From part (i) of Lemma \eqref{Lem:UnitFrechetOS} in conjunction with the continuous mapping theorem (or Mann-Wald theorem) it follows that: 
\begin{equation*}
    M^{(m)}_{k - [\theta_k/2]:k} \id V\Bigl( m (Z_{k - [\theta_k/2]:k} + 1/2) - m/2 \Bigr)\\
                           = V\Bigl( m \frac{2k}{\theta_k} - \frac{m}{2} \Bigr) \bigl( 1+ o_p(1) \bigr).
\end{equation*}
A direct consequence of Condition A1 is that $V(\infty) = \lim_{t\rightarrow \infty} V(t)= \infty$, through which statement we obtain $M^{(m)}_{k - [\theta_k/2]:k} \conv{P} \infty$, provided the intermediate sequence $\theta_{k}\rightarrow \infty$ and $\theta_{k}/k=o(1)$, as $k \rightarrow \infty$. The asymptotic development for the na\"{i}ve estimator \eqref{eq:EstNaive} will arise from Lemma \ref{Lem:UnitFrechetOS} (ii) as follows. By virtue of Condition A1,
\begin{eqnarray*}
 & &   \log M^{(m)}_{k - [\theta_k/4]:k} - \log M^{(m)}_{k - [\theta_k/2]:k}\\
   &\stackrel{d}{\approx}& \,\log V \Bigl( m \exp\bigl\{Z^{-1}_{[\theta_k/4]+1:k} \bigr\} - \frac{m}{2}\Bigr) - \log V \Bigl( m \exp\bigl\{Z^{-1}_{[\theta_k/2]+1:k} \bigr\} - \frac{m}{2}\Bigr)\\
 &=& \gamma \log \biggl( \frac{\exp\bigl\{Z^{-1}_{[\theta_k/4]+1:k} \bigr\} }{\exp\bigl\{Z^{-1}_{[\theta_k/2]+1:k} \bigr\} }\biggr)+ \gamma \log \biggl( \frac{1- 1/\bigl( 2\exp\bigl\{Z^{-1}_{[\theta_k/4]+1:k} \bigr\} \bigr)}{1- 1/\bigl( 2\exp\bigl\{Z^{-1}_{[\theta_k/2]+1:k} \bigr\}\bigr)} \biggr)\bigl( 1+ o_p(1) \bigr).
\end{eqnarray*}
Finally, application of Lemma \ref{Lem:UnitFrechetStdProcess} with $s=1/4, 1/2$ ascertains that
\begin{equation*}
    \log M^{(m)}_{k - [\theta_k/4]:k} - \log M^{(m)}_{k - [\theta_k/2]:k} = \gamma \log \Bigl( 2  + O_p\bigl( \frac{1}{\sqrt{\theta_k}} \bigr) \Bigr) + o_p(1).
\end{equation*}
\end{pfofThm}

%================== H2 Estimator =================================
%---------------------------
\section{Hybrid-Hill (H2) estimator based on $k_0$ tail-related BM}\label{Sec:HybridHill}

In this section, we investigate the asymptotic properties of the \citet{Hill1975} estimator, now cast at the intersection of BM and POT approaches for inference on a positive extreme value index $\gamma>0$. As with the cursory estimator \eqref{eq:EstNaive}, this  configures a semi-parametric approach,  amenable to a wide range of applications for its generality and data-driven simplicity.
Indeed, a key implication of Condition B, and one of impending practical consequence, is that
\begin{equation}
    \lim_{t\rightarrow \infty} \int_{1}^{\infty}\frac{1-F^m\bigl( x V(m(t-1/2)) \bigr) }{ 1 - F^m \bigl( V(m(t-1/2)) \bigr)}\, \frac{dx}{x} = \gamma,
\end{equation}
for $\gamma >0$. (cf. Lemma \ref{Lem:IntRVofFm} in Appendix \ref{App:CondB1}). By analogous reasoning to that which have led to the cursory estimator in Section \ref{Sec:CoEstim}, we proceed by replacing $t$ with $k/\theta_k \rightarrow \infty$, as $k \rightarrow \infty$, while also replacing the d.f. $F^m$ of the $k$ independent  BM with its empirical counterpart, namely the tail empirical distribution function given by
\begin{equation*}
    \overline{F_{k}^{m}}(x):= 1-F_{k}^{m}(x) = \frac{1}{k} \sumab{i=1}{k} \one_{\{M_{i}^{(m)} > x \}},
\end{equation*}
for all $x \in \real$. Adopting $\widehat{V}(m(k/\theta_k-1/2))=M^{(m)}_{k-k_0:k}$,  the hybrid-Hill estimator \eqref{eq:HibHill} thus arises from the integral (linear) functional representation:
\begin{equation*}
    \int_1^{\infty} \frac{1-F_{k}^{m}\bigl( x \widehat{V}(m(k/\theta_k-1/2)) \bigr) }{ 1-F_{k}^{m}\bigl( \widehat{V}(m(k/\theta_k-1/2)) \bigr)}\, \frac{dx}{x} = \int_{M^{(m)}_{k-k_0:k}}^{\infty} \frac{ \overline{F_{k}^{m}}(u)}{ \overline{F_{k}^{m}}\bigl(M_{k-k_0:k}\bigl) } \, \frac{du}{u} = \frac{k}{k_0} \int_{M^{(m)}_{k-k_0:k}}^{\infty} \bigl( 1-F_{k}^{m}(u)\bigr) \, \frac{du}{u}. 
\end{equation*}
Next, we establish its consistency and asymptotic normality in the framework issued by  Condition B and its second order strengthening borne out by A2.

\subsection{Consistency}

%--------- Theorem consistency of Hill for BM ---------
\begin{thm}\label{Thm:ConsistencyHybHill}
For any $m \geq 1$, assume that the distribution function $F^m$ underlying the BM sample satisfies Condition B. Suppose a sequence of positive constants $(\vartheta_k)$   satisfying $\sqrt{k}\vartheta^{-3}_k \rightarrow 0$, as $k \rightarrow \infty$.

\noindent If $k_0=k_0(k) \rightarrow \infty$ is such that $\vartheta_k\, k_0 = o(\sqrt{k})$, then
    \begin{equation*}
         \hat{\gamma}^{\textrm{H2}}\equiv \hat{\gamma}^{\textrm{H2}} (k_0, k) \conv{P} \gamma \,,
    \end{equation*}
    as $k \rightarrow \infty$.
\end{thm}
To prove Theorem \ref{Thm:ConsistencyHybHill} we first need a distribution-free result amenable to the universality class of distributions induced by the new extreme value condition \eqref{eq:RVofFm} encompassed in B. Not only does it set the foundations to the asymptotic properties of the H2 estimator, Lemma \ref{Lem:BasicEmpProcess} also draws the forethought on other possible unified approaches to the modelling of extreme value features in higher dimensions (recall discussion in Section \ref{Sec:ennes}).

\begin{lem}\label{Lem:BasicEmpProcess}
   For every $k\geq 2$ and constants $\vartheta_k > 1/2$, there exists a probability space where both the uniform empirical distribution function
\begin{equation*} 
	\frac{1}{k} \sumab{i=1}{k}\one_{\{U_{i}^{*}\leq s\}}
\end{equation*}
and the sample $ \{Z_{1},\ldots,Z_{k}\}$ of i.i.d unit Fr\'echet random variables with common distribution function $\exp(z^{-1})$, $z\geq0$, are defined and are such that 
    \begin{equation*}
            \suprem{(\vartheta_k (k+1))^{-1}< s \leq 1- \exp\{-(\vartheta_k \log k)^{-1} \} } \sqrt{k} \Bigl| Z_{k-[ks]:k}-\Bigl(\frac{1}{1-U^{*}_{k-[ks]:k}}- \frac{1}{2}\Bigr)\Bigl|\conv{a.s.} 0,
    \end{equation*}
provided that $\vartheta_k$ satisfies $\sqrt{k}\vartheta_k^{-3} \rightarrow 0$ and $k^{-1/2}\log \vartheta_k \rightarrow 0$, as $k \rightarrow \infty$.
\end{lem}
%--------- end of lemma------

\begin{rem}
	Condition $\sqrt{k}\vartheta_k^{-3}=o(1)$ imposes a lower bound to the sequence $(\vartheta_k)$ since it implies $\vartheta_k \rightarrow \infty$. Such a requirement can however be made weaker by imposing $\vartheta_k \log k \rightarrow \infty$ instead, aimed at establishing the above convergence to zero with probability tending to one, but with no guarantees that the process will inherit at least the same from the basic Fr\'echet quantile process than Lemma \ref{Lem:BasicEmpProcess} already accomplishes. % see text in page 23 of M. Csorgo's book, Quantile Processes with Statistical Applications. This page is bookmarked with a green post-it note.  
\end{rem}
The proof of Lemma \ref{Lem:BasicEmpProcess} is more of a probabilistic nature and, on account of this, is deferred to Section \ref{Sec:Proofs}.\\

%------- Proof of consistency ------------- 
\begin{pfofThm}~\ref{Thm:ConsistencyHybHill}:
Let  $Z_{1:k} \leq Z_{2:k}\leq \ldots \leq Z_{k:k}$ be the order statistics of the random sample $(Z_{1},\ldots,Z_{n})$ comprised of i.i.d. unit  Fr\'{e}chet random variables. In this instance, we develop a useful expansion for the regularly varying function $V$ intervening in the equality in distribution: \begin{equation*}
   \bigl\{ M^{(m)}_{k-i:k}\bigr\}_{i=0}^{k-1} \id \Bigl\{ V\Bigl(mZ_{k-i:k}\Bigr) \Bigr\}_{i=0}^{k-1}.
 \end{equation*}
Since $V \in RV_{\gamma}$, for some index $\gamma >0$, then so too $V(g(t))\in RV_{\gamma}$ with any positive function $g\in RV_{1}$, i.e. $V(g)$ satisfies Condition A1 locally uniformly. Specifically, we have for any $m\geq 2$ that
\begin{equation*}
  \frac{V\bigl(m(t+\frac{1}{2})\bigl)}{V(mt)} \Bigl(1-\frac{1}{2\bigl(t+\frac{1}{2}\bigl)}\Bigl)^{\gamma} =\frac{\Bigl(1-\frac{1}{2\bigl(t+\frac{1}{2}\bigl)}\Bigl)^{\gamma}V\bigl(m(t+\frac{1}{2})\bigl) }{V\bigl(m(t+\frac{1}{2}) (1-\frac{1/2}{t+1/2} )\bigl)}
  \arrowf{t} 1,
\end{equation*}
i.e., as $t \rightarrow \infty$,
\begin{equation}\label{eq:Aux2}
    \log V\Bigl(m\bigr(t+\frac{1}{2}\bigl)\Bigl)-\log V(mt)=-\gamma\log\Bigl(1-\frac{1}{2\bigl(t+\frac{1}{2}\bigl)}\Bigl)+o(1) =\frac{\gamma}{2}\Bigl(t+\frac{1}{2}\Bigl)^{-1} +o(t^{-1})+o(1).
\end{equation}
In line with the above, we write the log-spacings of the $k_0$-th BM as
 \begin{align*}
 	& \log M^{(m)}_{k-i:k} - \log M^{(m)}_{k-k_{0}:k}
 	\id \log V\bigl(mZ_{k-i:k}\bigl)-\log V\bigl(mZ_{k-k_{0}:k}\bigl)\\
   =& \log V\bigl(mZ_{k-i:k}+\frac{m}{2}\bigl)-\log V\bigl(mZ_{k-k_{0}:k}+\frac{m}{2}\bigl) -\Bigl[\log V\bigl(mZ_{k-i:k}
   +\frac{m}{2}\bigl)-\log V\bigl(mZ_{k-i:k}\bigl)\Bigl]\\
   & \quad +\Bigl[\log V\bigl(mZ_{k-k_{0}:k}+\frac{m}{2}\bigl)-\log V\bigl(mZ_{k-k_{0}:k}\bigl)\Bigl]\\
   =:&\;  I_{i} -II_{i} + III. 
\end{align*}
Heeding the intermediate sequence $k_{0}= k_0(k)$, we have from (\ref{eq:Aux2}) in conjunction with Lemma \ref{Lem:UnitFrechetOS} (i) (the latter ensures $Z_{k-k_{0}:k}\conv{P}\infty$) that both $II$ and $III$ reduce to a $o_{p}(1)$-term. The term $I_{i}$, for every $i=0, 1, \ldots, k_0-1$, summarises the meaningful random element to the asymptotics of the H2 estimator. Applying Condition A1(i), it follows that
\begin{equation*}
   I_i =\log V\Bigl(\frac{Z_{k-i:k}+\frac{1}{2}}{Z_{k-k_{0}:k}+\frac{1}{2}}m\bigl(Z_{k-k_{0}:k}+\frac{1}{2}\bigl)\Bigl)-\log V\Bigl(mZ_{k-k_{0}:k}+\frac{m}{2}\Bigl)
   =\gamma\log\Bigl(\frac{Z_{k-i:k}+\frac{1}{2}}{Z_{k-k_{0}:k}+\frac{1}{2}}\Bigl)(1+o_{p}(1)).
\end{equation*}
There are accompanying uniform bounds, encompassing Condition A1, which we now put into use for making the latter $o_{p}(1)$-term more precise: for arbitrary $\varepsilon$, $\delta>0$, there exists $t_{0}=t_{0}(\varepsilon,\delta)$ such that for $x\geq 1$ and $t\geq t_{0}$,
\begin{equation*}
     \log(1-\varepsilon)+(\gamma-\delta)\log x<\log V \bigl(tx\bigl) - \log V\bigl(t\bigr)<(\gamma+\delta)\log x + \log(1+\varepsilon).
\end{equation*}
We are going to employ these inequalities with $x =(Z_{k-i:k}+\frac{1}{2})/(Z_{k-k_{0}:k}+\frac{1}{2})$ -- which is greater than $1$ with probability $1$ -- and $t= m\bigl(Z_{k-k_{0}:k}+\frac{1}{2}\bigl)\conv{p} \infty$ (via Lemma \ref{Lem:UnitFrechetOS} part (i)), so that eventually, 
\begin{equation*}
     \log(1-\varepsilon)-\delta\log\Bigl(\frac{Z_{k-i:k}+\frac{1}{2}}{Z_{k-k_{0}:k}+\frac{1}{2}}\Bigl) <I_{i}-\gamma\log\Bigl(\frac{Z_{k-i:k}+\frac{1}{2}}{Z_{k-k_{0}:k}+\frac{1}{2}}\Bigl)<\delta\log\Bigl(\frac{Z_{k-i:k}+\frac{1}{2}}{Z_{k-k_{0}:k}+\frac{1}{2}}\Bigl)+ \log(1+\varepsilon);
\end{equation*}
 followed by application of the primary Lemma \ref{Lem:BasicEmpProcess} in order to get the upper bound
\begin{equation}\label{eq:expandConsistency}
     \frac{1}{k_{0}}\sum_{i=0} ^{k_{0}-1}I_{i}-\gamma <\gamma\Bigl(\frac{1}{k_{0}}\sum_{i=0} ^{k_{0}-1}\log\frac{1-U^{*}_{k-k_{0}:k}}{1-U^{*}_{k-i:k}}-1\Bigl)+\frac{\delta}{k_{0}}\sum_{i=0} ^{k_{0}-1}\log\frac{1-U^{*}_{k-k_{0}:k}}{1-U^{*}_{k-i:k}}+\log (1+\varepsilon),
\end{equation}
for an intermediate sequence $k_0 \rightarrow \infty$ and $k_0/k \rightarrow 0$, as $k \rightarrow \infty$, in the conditions stipulated in the theorem (see Remark \ref{Rem:Discussion} below for a brief note in this respect).
A completely analogous lower bound guarantees tightness. Owing to Malmquist's result \citep[see][p.41]{Reiss1989}, application of the law of large numbers leads to
\begin{equation*}
  \frac{1}{k_{0}}\sum_{i=0} ^{k_{0}-1}\log\frac{1-U^{*}_{k-k_{0}:k}}{1-U^{*}_{k-i:k}} \, \id \, \frac{1}{k_{0}} \sum_{i=0} ^{k_{0}-1} \bigl( - \log (1- U^{*}_{i:k_0}) \bigr) \conv{P}\intab{1}{\infty}\log y\, \frac{dy}{y^{2}}=1 \end{equation*}
as $k_0 \rightarrow \infty$. Therefore, from \eqref{eq:expandConsistency} we arrive at 
\begin{equation*}
    \Bigl|\frac{1}{k_{0}}\sum_{i=0} ^{k_{0}-1}I_{i}-\gamma\Bigl|<\delta \pm \log(1+\varepsilon')+o_{p}(1).
\end{equation*}
This is in keeping with Lemma 3.2.3 in \citet{deHaanFerreira2006}, in which result both $II_i$ and $III$ are also deemed lower-order terms. Hence, the proof is complete.
\end{pfofThm}
\begin{rem}\label{Rem:Discussion}% page 27 in my notes
\begin{comment}
	On the other hand, it is reasonable to assume that $k_0/k$ is bounded from above by $\beta_k = 1-\exp\{-1/(\vartheta_k \log k)\}$ tending to zero as part of the conditions of the theorem (primarily dictated by Lemma \ref{Lem:BasicEmpProcess}), whereby $k_0$ becomes an intermediate sequence such that  $k_0 \rightarrow \infty$ and $k_0=o(k)$. This translates as $\vartheta_k\,k_0 \sim c\, k/\log k$, for some constant $c \in (0,1] $, which is a weaker assumption than that set out as part of Theorem \ref{Thm:ConsistencyHybHill} (i.e., addressing the case $c=0$).}
\end{comment}
The slightly stronger constraint $\vartheta_k k_0 / \sqrt{k} \rightarrow 0$ has been put in place to ensure that the actual upper bounds $(\vartheta_k)$ demanded, on the one hand by Lemma \ref{Lem:BasicEmpProcess} (that   $k^{-1/2} \log \vartheta_k \rightarrow 0$), and by the consistency argument in the theorem on the other hand (specifically that $k_0/k = o(1 /(\vartheta_k \log k))$) are not contradictory of one another. Noticeably, the second restriction implies $k_0/k =o(1)$.
\end{rem}

\subsection{Asymptotic normality}

This section is primarily aimed at the asymptotic normality for the proposed hybrid-Hill estimator. A secondary objective is to lay the groundwork for the reduced-bias estimation developed later on in Section \ref{Sec:RBH2}. The exposition in the proofs as part of this section is deliberately aimed at making it obvious that our methods apply to virtually any existing semi-parametric methodology for estimation and testing concerning extreme domains of attraction \citep[see][]{Embrechtsetal1997,deHaanFerreira2006,Neves2009}. Combined with the universality class induced by relation \eqref{eq:RVofFm}, Lemma \ref{Lem:BasicEmpProcess} turns out to be pivotal towards delivering this aim: it warrants the notorious brevity of this section as the proofs are made purposefully similar to those in \citet[][Chapter 3]{deHaanFerreira2006} from the point of fusion between BM and POT domains.

\begin{thm}\label{Thm:AN}
   Assume Conditions B and A2 and define $g(t):= t-1/2$. Then, for all $m= 1, \ldots, n$, there exists a function $\tilde{A}$, ultimately of constant sign and satisfying $\lim_{t \rightarrow \infty}\tilde{A}(t)=0$, such that
 \begin{equation}\label{2RVlogVg}
 \limit{t} \frac{\log V(mg(tx))-\log V(mg(t))-\gamma\log x}{\tilde{A}(mt)}= 
 \frac{x^{\tilde{\rho}}-1}{\tilde{\rho}},
\end{equation}
for all $x>0$, locally uniformly. Moreover, $|\tilde{A}| \in RV_{\tilde{\rho}}$ with  $\tilde{\rho}:=\max( \rho,-1)$. Let $k_0=k_0(k)$ in the conditions of Theorem \ref{Thm:ConsistencyHybHill} be an upper intermediate sequence, i.e. $k_0 \rightarrow \infty$ and $k_0/k \rightarrow 0$, as $k\rightarrow\infty$, satisfying
\begin{equation*}
     \lim\limits_{k\to \infty} \sqrt{k_0}\tilde{A}\Bigl(\frac{n}{k_0}\Bigl)=\lambda \in \real.
\end{equation*}
Then, the H2 estimator is asymptotically normal, namely
\begin{equation*}
         \sqrt{k_0}\bigl(\hat{\gamma}^{\textrm{H2}}-\gamma \bigl)\conv{d}N\Bigl(\frac{\lambda}{1-\tilde{\rho}},\gamma^2\Bigl).
\end{equation*}
\end{thm}
\begin{rem}
	The asymptotic normality of $\hat{\gamma}^{\textrm{H2}}$ does not depend explicitly on the block length $m$ since it only affects the second order convergence relating the approximation bias. 
\end{rem}

%-------- Proof of Theorem AN ---------
\begin{proof}
The first part of the theorem follows readily from Lemma \ref{Lem:2ndOrdLogShift} in Appendix \ref{App:CondB2}. Next, we somewhat departure from the line of proof for consistency of $\hat{\gamma}^{\textrm{H2}}$ (given in Theorem \ref{Thm:ConsistencyHybHill}) by shifting the focus from the prior Fr\'echet order statistics
%
\begin{comment}
	, as in the representation
\begin{equation*}
	\hat{\gamma}^{\textrm{H2}}= \frac{1}{k_0}\sumab{i=0}{k_0-1} \log M^{(m)}_{k - i:k} - \log M^{(m)}_{k -k_0:k} \, \id \, \frac{1}{k_0}\sumab{i=0}{k_0-1}\log V\bigl(mZ_{k-i:k}\bigl)-\log V\bigl(mZ_{k-k_{0}:k}\bigl),
\end{equation*}
\end{comment}
%
to the dual representation in terms of Pareto supplied by Lemma \ref{Lem:BasicEmpProcess}. Loosely speaking, Lemma \ref{Lem:BasicEmpProcess} states that, from some given high threshold $t$ onwards, the conditional distribution function of the shifted standard Pareto random variable $Y-1/2$ is well approximated by the tail distribution function of the unit Fr\'echet distribution.

\noindent Let $Y_{1:k} \leq Y_{2:k} \leq \dots \leq Y_{k:k}$ be the order statistics from a sample of i.i.d standard Pareto random variables with tail d.f. given by $1 \wedge y^{-1}$. Using Lemma \ref{Lem:BasicEmpProcess}, for the same intermediate sequence $k_0 \in (0, k]$, we recast the above equality in distribution into the almost sure approximation,
\begin{equation*}
	\sqrt{k_0} (\hat{\gamma}^{\textrm{H2}} - \gamma)  \stackrel{a.s.}{\approx} \, \sqrt{k_0} \Bigl\{ \frac{1}{k_0}\sumab{i=0}{k_0-1} \log V  \Bigr(m\bigl(\frac{Y_{k-i:k} }{Y_{k-k_{0}:k}} Y_{k-k_{0}:k}-\frac{1}{2}\bigl)\Bigr) - \log V \Bigl( m \bigl(Y_{k-k_{0}:k} - \frac{1}{2}\bigl) \Bigr)- \gamma\Bigr\}.
\end{equation*}
The assumption in the theorem that $\vartheta_k k_0/\sqrt{k} \rightarrow 0$, as $k \rightarrow \infty$, implies $k_0/k \rightarrow 0$ whereby \ref{Lem:UnitFrechetOS} (ii) ensures that $(k_0/k) Y_{k-k_0:k} \conv{P}1$. Since $ \{Y_{k-k_0:k}\}$ is a non-decreasing sequence w.r.t. $k$, the convergence in probability $Y_{k-k_0:k} \conv{P} 1$ implies convergence almost surely.
Via analogous uniform bounds to \eqref{inequality} from Condition A2 now within the realm of \eqref{2RVlogVg}, i.e.
\begin{equation*}
 \suprem{x\geq 1} \, x^{-(\tilde{\rho}+\delta)} \biggl|\Bigl(\frac{\log V(mg(tx))-\log V(mg(t))-\gamma \log x}{\tilde{A}(mt)}-\frac{x^{\tilde{\rho}}-1}{\tilde{\rho}}\Bigr)\biggr|\leq\varepsilon \,,
\end{equation*}
we arrive at
\begin{eqnarray}
\nonumber	 & & \sqrt{k_0} \Bigl\{ \frac{1}{k_0}\sumab{i=0}{k_0-1} \gamma \log \frac{Y_{k-i:k} }{Y_{k-k_{0}:k}} -\gamma + \tilde{A}(m Y_{k-k_{0}:k}) \Bigl[ \frac{1}{k_0}\sumab{i=0}{k_0-1}  \frac{(Y_{k-i:k} /Y_{k-k_{0}:k})^{\tilde{\rho}}-1}{\tilde{\rho}}\\
\nonumber  & & \mbox{\hspace{9cm}}+ \frac{1}{k_0}\sumab{i=0}{k_0-1} \Bigl( \frac{Y_{k-i:k} }{Y_{k-k_{0}:k}} \Bigr)^{\tilde{\rho}+ \varepsilon} \Upsilon_k(i)   \Bigr] \Bigr\}\\
\nonumber & \leq & \sqrt{k_0} \Bigl\{ \gamma \Bigl(\frac{1}{k_0}\sumab{i=0}{k_0-1} \log \frac{Y_{k-i:k} }{Y_{k-k_{0}:k}} -1 \Bigr) + \tilde{A}(m Y_{k-k_{0}:k}) \Bigl[ \frac{1}{k_0}\sumab{i=0}{k_0-1}  \frac{(Y_{k-i:k} /Y_{k-k_{0}:k})^{\tilde{\rho}}-1}{\tilde{\rho}}\\
  & & \mbox{\hspace{6.5cm}}  + \max_{1\leq i \leq k_0} \bigl|\Upsilon_k(i-1)\bigr| \frac{1}{k_0}\sumab{i=0}{k_0-1} \Bigl( \frac{Y_{k-i:k} }{Y_{k-k_{0}:k}} \Bigr)^{\tilde{\rho}+ \varepsilon}    \Bigr] \Bigr\}, \label{eq:UBoundRand}
\end{eqnarray}
where, for every $i=0, 1, \ldots, k_0-1$,
\begin{equation*}
	\Upsilon_k(i):= \Bigl( \frac{Y_{k-i:k} }{Y_{k-k_{0}:k}} \Bigr)^{-\tilde{\rho}- \varepsilon}\Bigl( \frac{\log V  \bigl(m(Y_{k-i:k} -\frac{1}{2})\bigr) - \log V \bigl( m (Y_{k-k_{0}:k} - \frac{1}{2}) \bigr)}{\tilde{A}(m Y_{k-k_{0}:k})}- \frac{(Y_{k-i:k} /Y_{k-k_{0}:k})^{\tilde{\rho}}-1}{\tilde{\rho}} \Bigr)= o_p(1).
\end{equation*}
A similar lower bound can be readily established. The first term in \eqref{eq:UBoundRand} is $\sqrt{k_0}$-consistent and asymptotically normal by Lemma 3.2.3 of \citet{deHaanFerreira2006}. For the remainder terms attached to $\tilde{A}(m Y_{k-k_{0}:k})$, the law of large numbers ascertains the result in the theorem by similar steps as in the proof of Lemma 3.2.3 of \citet{deHaanFerreira2006}. In particular, for any arbitrarily small $c\geq 0$,
\begin{equation*}
	\frac{1}{k_0}\sumab{i=0}{k_0-1} \Bigl( \frac{Y_{k-i:k} }{Y_{k-k_{0}:k}} \Bigr)^{\tilde{\rho}+ c}  \conv{P} \intinf{1} y^{\tilde{\rho}+ c} \, \frac{dy}{y^2} = \frac{1}{1-\tilde{\rho} - c}\, ,
\end{equation*}
as $k_0 \rightarrow \infty$. It only remains to prove that
$
     \tilde{A}(mY_{k-k_0:k})/\tilde{A}\bigl(m\frac{k_0}{k}\bigl) \conv{P} 1.
$
Lemma \ref{Lem:UnitFrechetOS} (i) for $\theta_k=k_0$ in conjunction with the regular variation $\lim_{t \rightarrow \infty}\tilde{A}(tx)/\tilde{A}(t) = x^{\tilde{\rho}}$, for $x>0$ locally uniformly gives that
\begin{equation*}
        \Bigl|\tilde{A}(mY_{k-k_0:k})\Bigl| \,\approxid\,  \Bigl(\frac{k_0}{k}(Z_{k-k_{0}:k}+1/2)\Bigl)^{\tilde{\rho}} \Bigl|\tilde{A}\Bigl(m\frac{k}{k_0}\Bigl)\Bigl|(1+o_{p}(1))=\Bigl|\tilde{A}\Bigl(m\frac{k}{k_0}\Bigl)\Bigl|\Bigl(1+o_{p}\bigl(\frac{1}{\sqrt{k_0}} \bigr)\Bigr).
\end{equation*}
Hence the asymptotic expansion for the normalised H2 estimator:
\begin{equation}\label{eq:H2FullExpansion}
     \sqrt{k_0} (\hat{\gamma}^{\textrm{H2}}-\gamma)\id \ \gamma P_{k_0}+\frac{1}{1-\tilde{\rho}}\sqrt{k_0}\tilde{A}\Bigl(m\frac{k}{k_0}\Bigl)+o_{p}\Bigr(\sqrt{k_0}\tilde{A}\Bigl(m\frac{k}{k_0}\Bigl)\Bigr),
   \end{equation}
where $P_{k_0}$ denotes standard normal random variable. This expansion, with wide resonance in POT-related estimation methodology \citep[cf.][Chp. 3]{deHaanFerreira2006}, concludes the proof.
\end{proof}

%============== Beginning of Optimal selection k0 ============
\section{Optimal choices for the top fraction $k_0/k$ in H2 estimation} \label{Sec:OptimalFraction}

In this section, we explore the optimal selection of $k_{0}=k_{0}(k)$ in the sense of minimising the asymptotic mean squared error (AMSE) of the H2 estimator. Under the conditions of Theorem \ref{Thm:AN}, we wish to minimise the AMSE$= E_{\infty}[k_0( \hat{\gamma}^{\textrm{H2}}(k_0,k)-\gamma)^2]$ with respect to $k_0$. This procedure will be hinged on the approximation originating from \eqref{eq:H2FullExpansion}:
 \begin{equation*}
              E\bigl[k_0( \hat{\gamma}^{\textrm{H2}}(k_0,k)-\gamma)^2\bigr]\, \approx \, \gamma^2 Var(P_{k_{0}}) + E^2\Bigl[P_{k_{0}}+ \frac{1}{1-\tilde{\rho}}\sqrt{k_0} \tilde{A}\bigl(m\frac{k}{k_0} \bigl)\Bigr] = \gamma^2 + \Bigl(\frac{\sqrt{k_0}}{1-\tilde{\rho}} \tilde{A}\bigl(m\frac{k}{k_0} \bigl)\Bigr)^2.
 \end{equation*}
In the subsequent search for a minimiser of $\sqrt{k_0} \tilde{A}\bigl(m\nicefrac{k}{k_0})$,
two cases are in focus which we address separately: the case of rapid convergence determined by $\tilde{\rho}= -1$ and the slower convergence sub-class of $-1< \tilde{\rho} \leq 0$. We note that within the former, the case $\tilde{\rho}=-1$ and $\rho =-1$ is approached by continuity (cf. Appendix \ref{App:CondB2}). The supplementary document to this paper lists a number of distributions and their indices of regular variation, which can serve as an aid to what comes next.
\begin{itemize}
    \item For $\tilde{\rho}=-1$, the deterministic component of the asymptotic bias is primarily determined by
\begin{equation*}
          b_k= \frac{\gamma}{4}\sqrt{k_0}\Bigl|\Bigl(\frac{k}{k_0}-\frac{1}{2}\Bigl)^{-1}\Bigr|,
\end{equation*}
and becomes negligible for $k_0 \ll k^{2/3}$. Given the wide range of distributions possessing $\tilde{\rho}= -1$, including Fr\'echet, Pareto, Cauchy as well as quite a few Burr distributions, this upper bound for enabling asymptotically null bias is in a certain sense distribution-free. A list of models with $\tilde{\rho}= -1$ is given in Table 1, in the supplement.

\item  For $-1<\tilde{\rho}\leq 0$, the interest is in finding the largest possible $k_0$ (hence minimising the variance) such that $\sqrt{k_0} \tilde{A}(n/k_0)$ is arbitrarily close to zero. We define $L(t):= t^{-\tilde{\rho}} \tilde{A}(t) \in RV_0$, whereby there exists a positive function $a$ such that $t^{-1}a(t) \in RV_{-1}$ and
\begin{equation*}
	L(t) = \intinf{t} \frac{a(s)}{s}\, ds.
\end{equation*}
Invoking Karamata's theorem for integration of regularly varying functions \citep[][Theorem B.1.5]{deHaanFerreira2006}, we find that
\begin{equation*}
	\frac{L(t)}{a(t)} = \frac{1}{t} \intinf{t} \frac{s^{-1}a(s)}{t^{-1}a(t)}\, ds \arrowf{t} \infty.
\end{equation*}
This implies that whenever the slowly-varying part of $\tilde{A}$ satisfies $L(t) \rightarrow c \in (0,\infty)$, minimisation through finding the zeros of its derivative is rendered useless due to the necessarily $a(t)=o(1)$. That is, $a(t) \ll L(t)$ meaning that the works on $a$ is not enough to keep $L$ in check, eventually. This partly explains the excessively large turning points reported in \citet[][Section 3]{CN2004}. Hence, we proceed via the direct method and write $t_k= k/k_0 \rightarrow \infty$ and $\sqrt{k_0} \tilde{A} (t_k) = \sqrt{k}\, t_k^{-1/2+ \tilde{\rho}}L(t_k)$. If $\lim_{t\rightarrow \infty}L(t)$ is positive and finite, then we may set $ \sqrt{k}\, t_k^{-1/2+ \tilde{\rho}} \rightarrow \beta \in [0,1)$ which, solving for $k_0$, yields $k_0 \sim \beta^{2/(1-2\tilde{\rho})}k^{1-1/(1-2\tilde{\rho})} \rightarrow \infty$. The role of the block size $m$ is that of a catalyst in the convergence whereby we can replace $k$ with $mk$ thus making it possible to conduct the H2 estimation on a wider range of $k_0$ in a optimal way. We note that if $\tilde{\rho}= -1$, we obtain $k_0 \sim (\beta k)^{2/3}$, the same power of $k$ that staves off the bias in the previous item. If $L(t) =o(1)$, as $t\rightarrow \infty$, then we seek the largest $k_0$ such that $t_k^{1/2+ \tilde{\rho}}/(mk)^{-1/2} \rightarrow 0$, i.e. $ k_0 \ll (mk)^{-2\tilde{\rho}/(1-2\tilde{\rho})}$. As it tallies  with the previously obtained $k_0 \ll (mk)^{2/3}$, it establishes continuity with the case of $\tilde{\rho}=-1$.
\end{itemize}

%=============== Beginning of Section Reduced Bias =================
\section{Reduced-bias H2 estimation}\label{Sec:RBH2}

The primary goal in this section is to provide substantive evidence that the deterministic, second order, approximation bias embedded in the H2 estimator can be curtailed in a systematic manner. The cursory estimator, if used as a surrogate estimator for $\gamma>0$ that features in this approximation bias alone, is found largely useful to the reduced-bias H2 estimation that we carve out next. In particular, it will circumvent the usual requirement of having to conduct estimation of the second order elements at a much lower threshold, usually advised of the order $2k/\log \log k$.

Key to curbing the asymptotic bias of the H2 estimator is the worsened rate of convergence displayed in the second order relation \eqref{2RVlogVg}, whose second order parameter $\tilde{\rho}$ becomes trapped at $-1$ for a wide range of distributions $F$ possessing $V \in 2RV(\gamma, \rho)$, $\rho \leq 0$. We have opted not to delve further into the asymptotic normality since bias-reduction has become a research topic in its own right and straying away to such endeavour would lengthen the paper considerably and would risk detracting on its chief purpose, that  of blending the two (BM and POT) streams for inference on extremes.

Figures \ref{Fig:RBFrechetPareto025} to \ref{Fig:RBGEVGPD025} display the estimated bias and empirical mean squared error (MSE) of the H2 estimator in \eqref{eq:HibHill} and of its reduced-bias variants in both cases of $\tilde{\rho} = -1$ (labelled RBH2r) and $\tilde{\rho}> -1$ (labelled RBH2). For each model, the simulations essentially consisted of generating a sample of size $n=5000$ which is subsequently segmented into blocks of size $n=20$. Then H2's estimates are obtained for every increasing number $k_0$ of larger BM  whereas maximum likelihood (ML) estimation is computed with $k_0=k$, by assuming that the whole sample of BM conforms to the exact GEV-fit. In order to obtain estimates of H2 shorn of its bias, we employ the cursory estimator  $\hat{\gamma}(\theta_k, k)$  to the estimation of the unknowns embedded in $\tilde{A}(k/k_0)$ from \eqref{eq:H2FullExpansion} (specified in Appendix \ref{App:CondB2}). This task is much facilitated by the many parent distributions possessing $\tilde{\rho} = -1$. A number of these are shown in Table 1 as part of the accompanying supplement.

In all cases encompassing Figures \ref{Fig:RBFrechetPareto025} and \ref{Fig:RBBurr}, the ML estimator quite quickly gets overcome by a H2 estimator either in terms of bias or of MSE, or both. For the Burr distribution satisfying \eqref{2RVlogVg} with $\gamma=0.3$ and $\tilde{\rho}= -1$ (details in Table 2 in the supplement), the reduced-bias H2 estimation yields sustained gains in efficiency for a wide range of $k_0$. %, however slightly less precise than the MLE. This reflects the larger  asymptotic variance of the RBH2r estimator for $\gamma \geq 1$. Although such values are hardly the case in practical applications, especially in environmental ones since $\gamma>1$ signifies that only moments of a fractional order exist finite, we believe that further improvement of RBH2r is within reach.
For enhanced clarity, we are also illustrating performance of the H2 estimator with the GEV and GP distributions, both having their second order parameter $\tilde{\rho}= -\gamma$ estimated with the cursory estimator. In spite of the MLE being at greater advantage in this setting, not least because it benefits from being invariant to a location-transform in the data, the reduced-bias RBH2 remains a convincing competitor for small values of $k_0$.

\begin{figure}
\begin{center}
\includegraphics[scale=0.4]{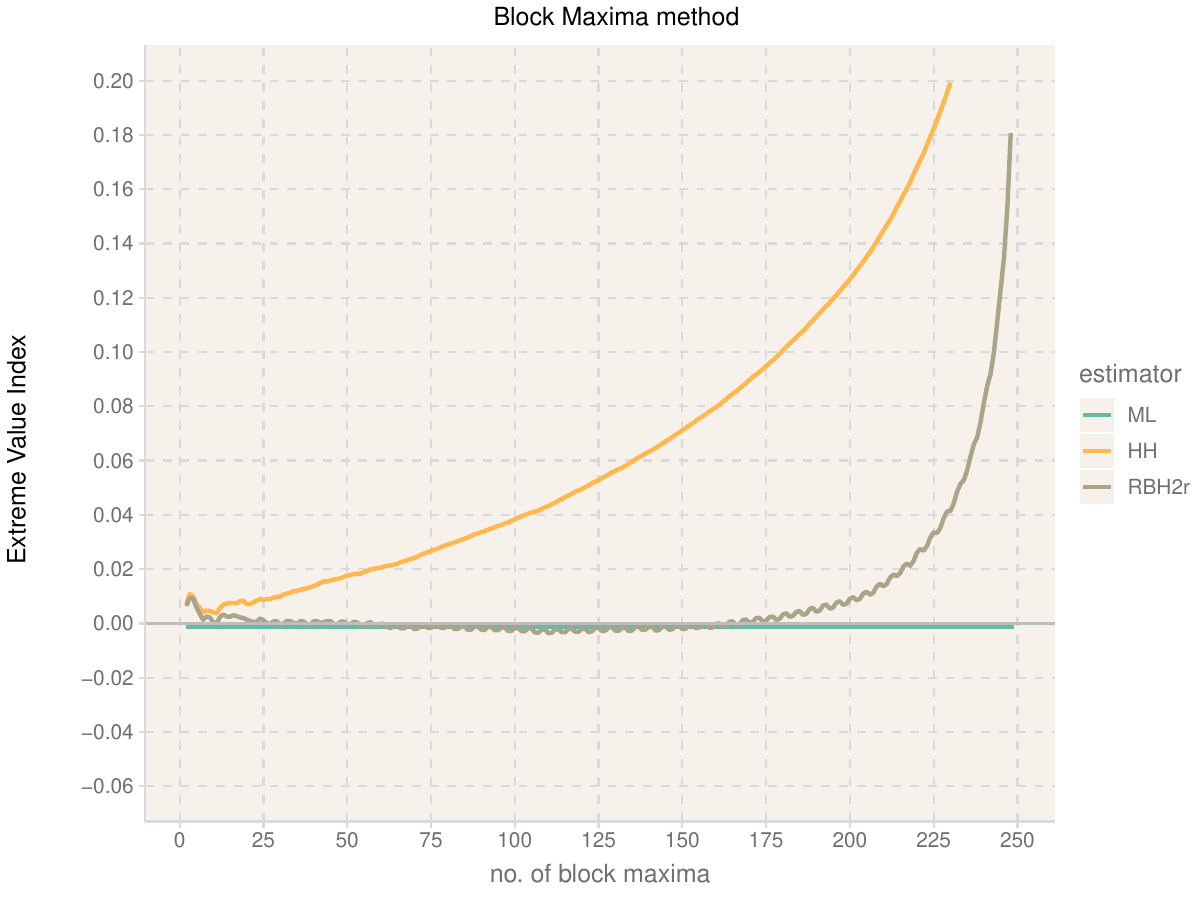} 
\includegraphics[scale=0.4]{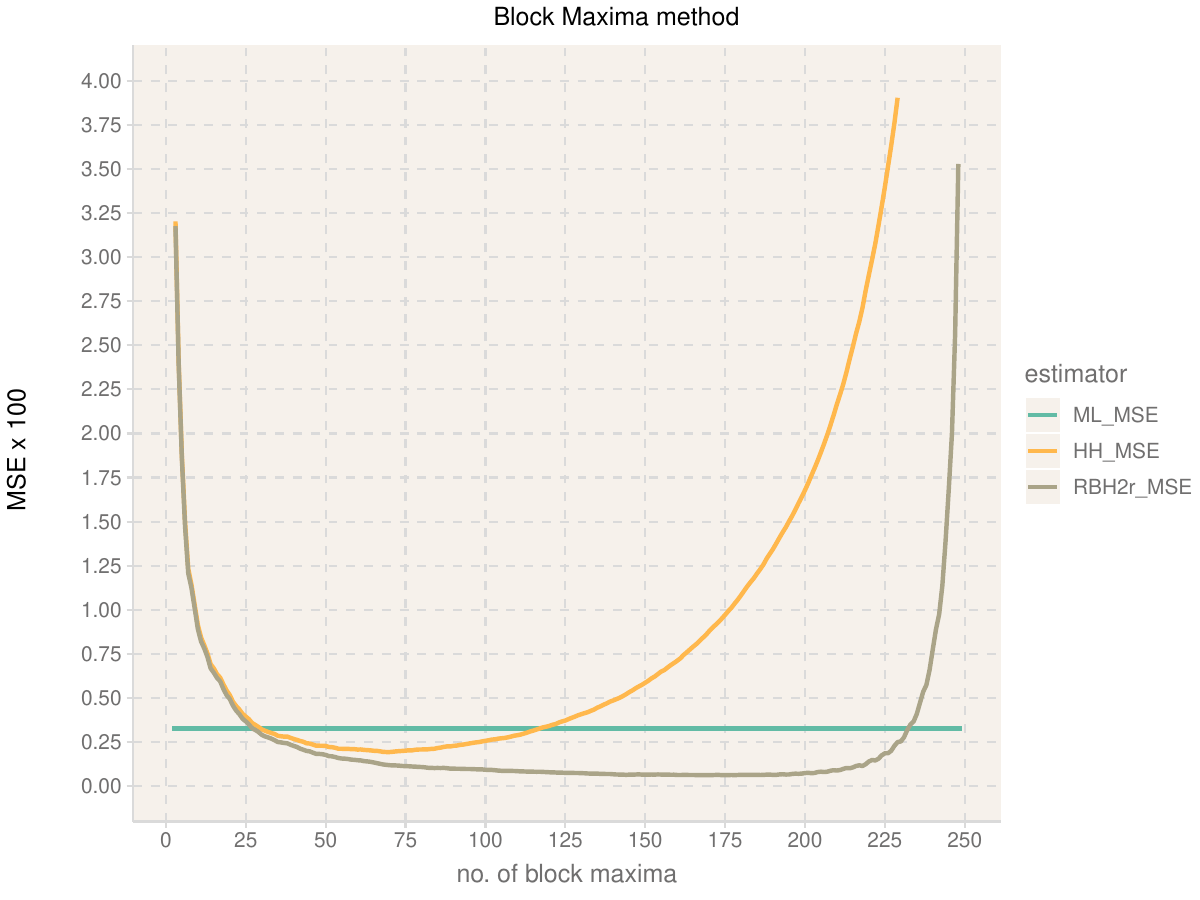}
\includegraphics[scale=0.4]{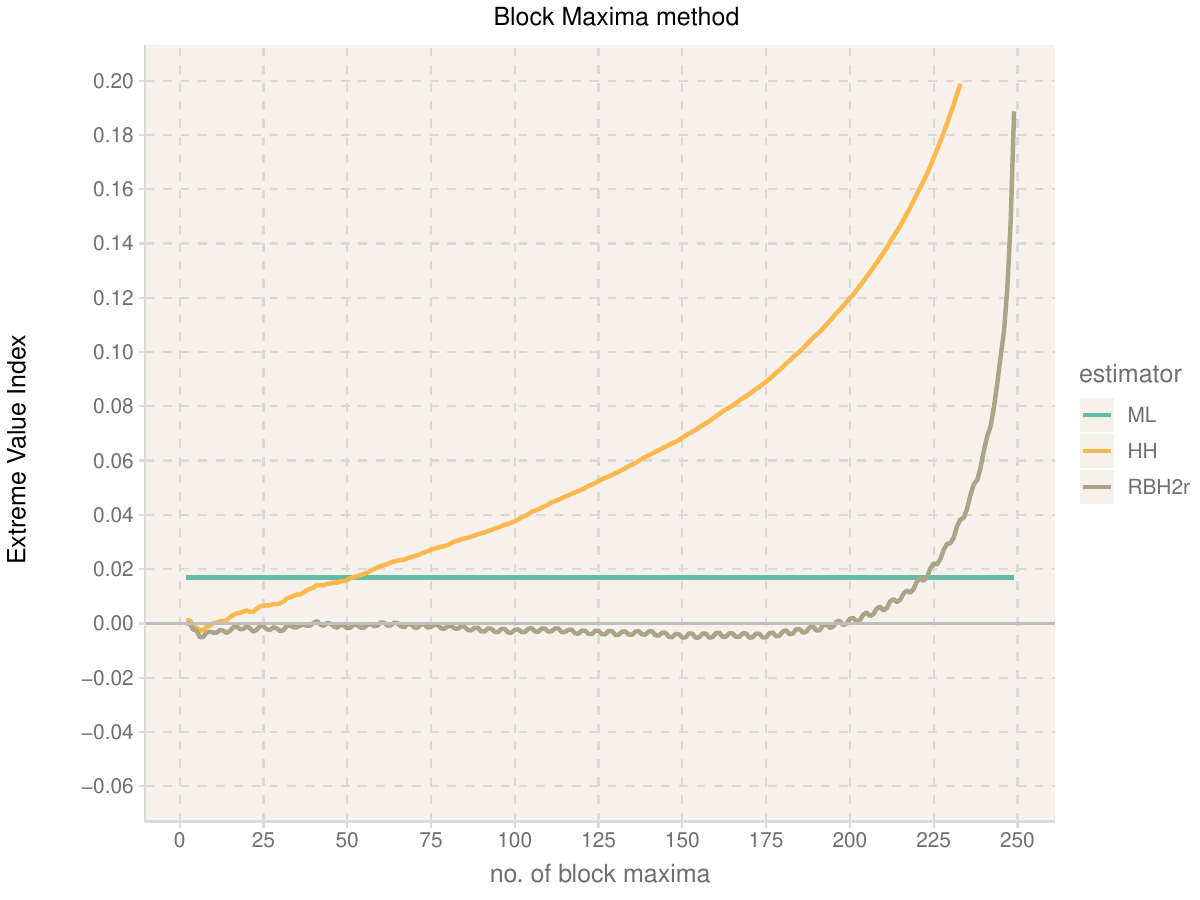} 
\includegraphics[scale=0.4]{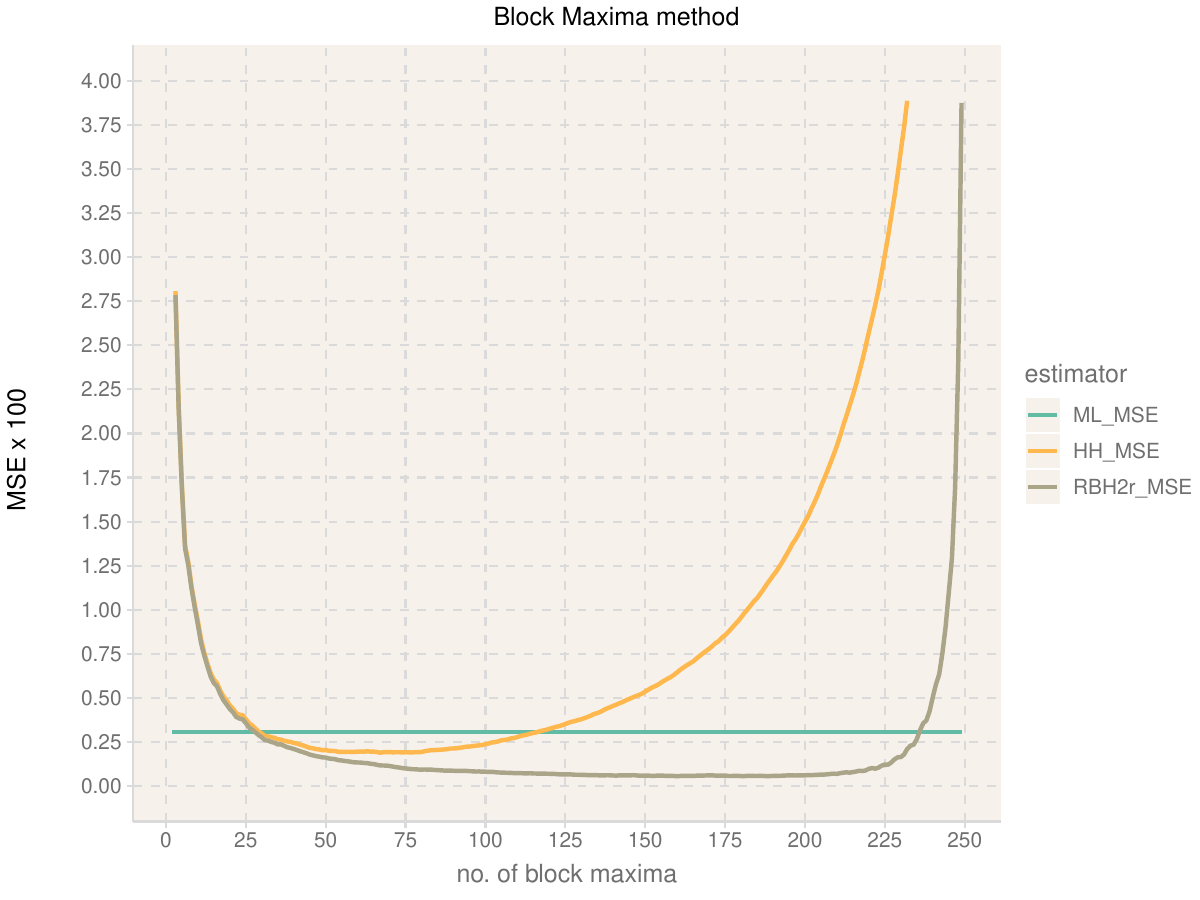}
\end{center}
\caption{Fr\'echet and Pareto parent distributions, both attached to $\gamma=0.3$ and $\tilde{\rho}= -1$. Inference produced by H2 estimators is virtually the same for these two distributions. }
\label{Fig:RBFrechetPareto025}
\end{figure}

\begin{figure}
\begin{center}
\includegraphics[scale=0.4]{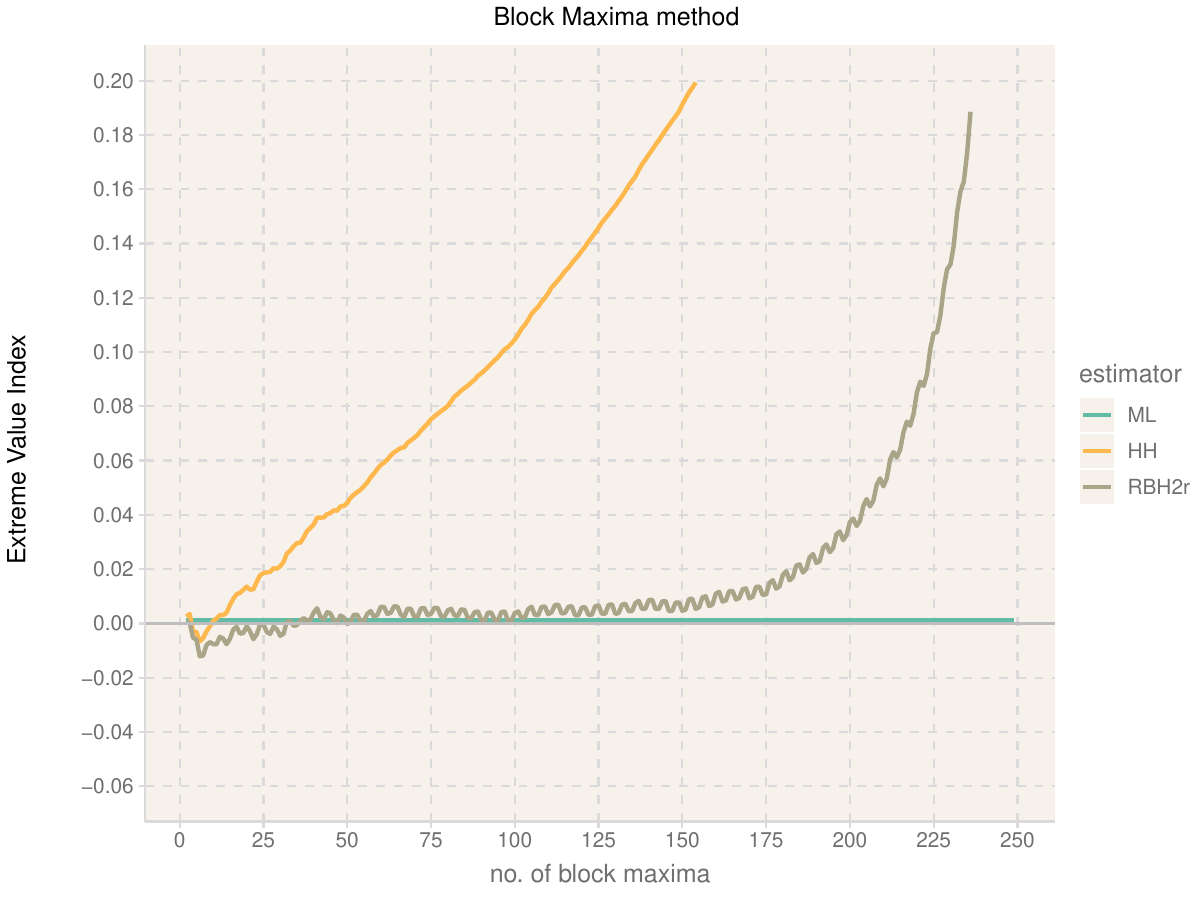} 
\includegraphics[scale=0.4]{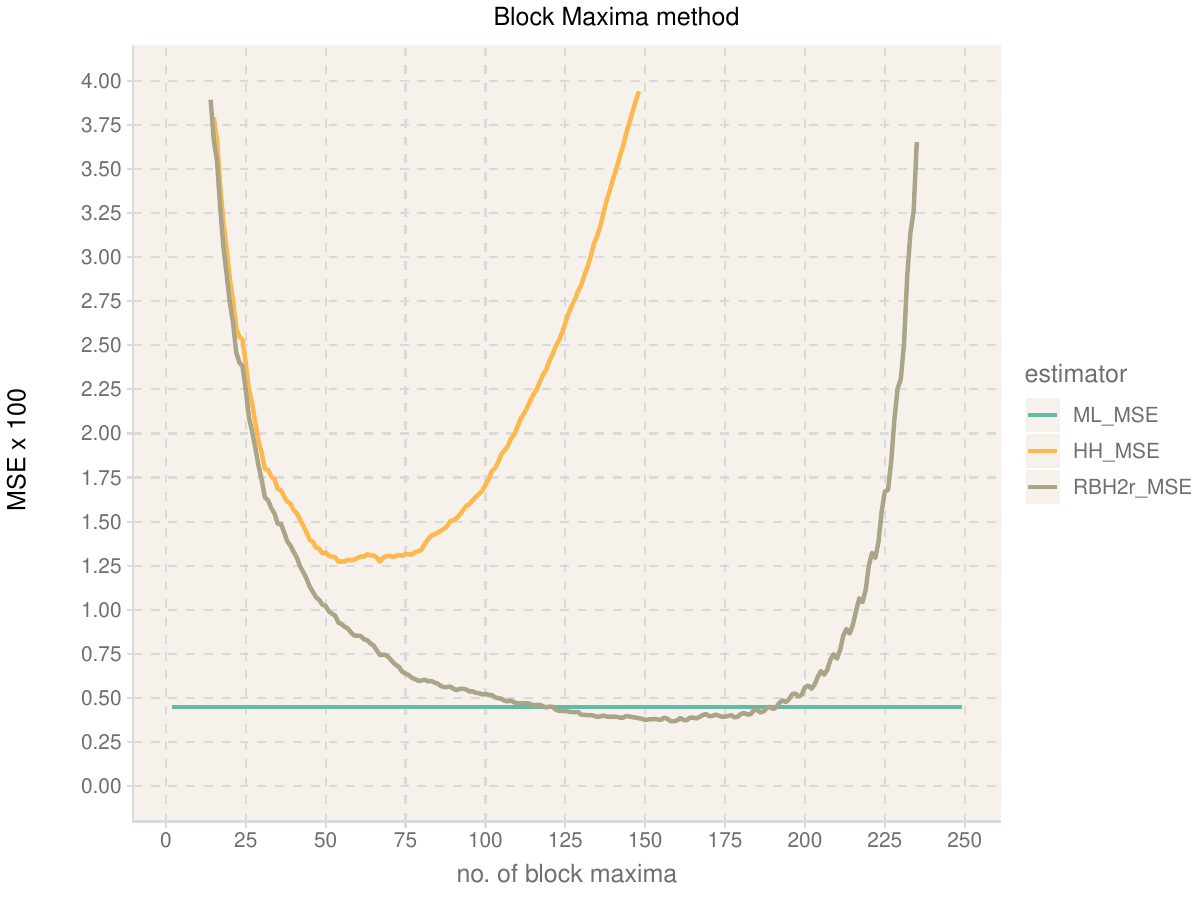}\\
\includegraphics[scale=0.4]{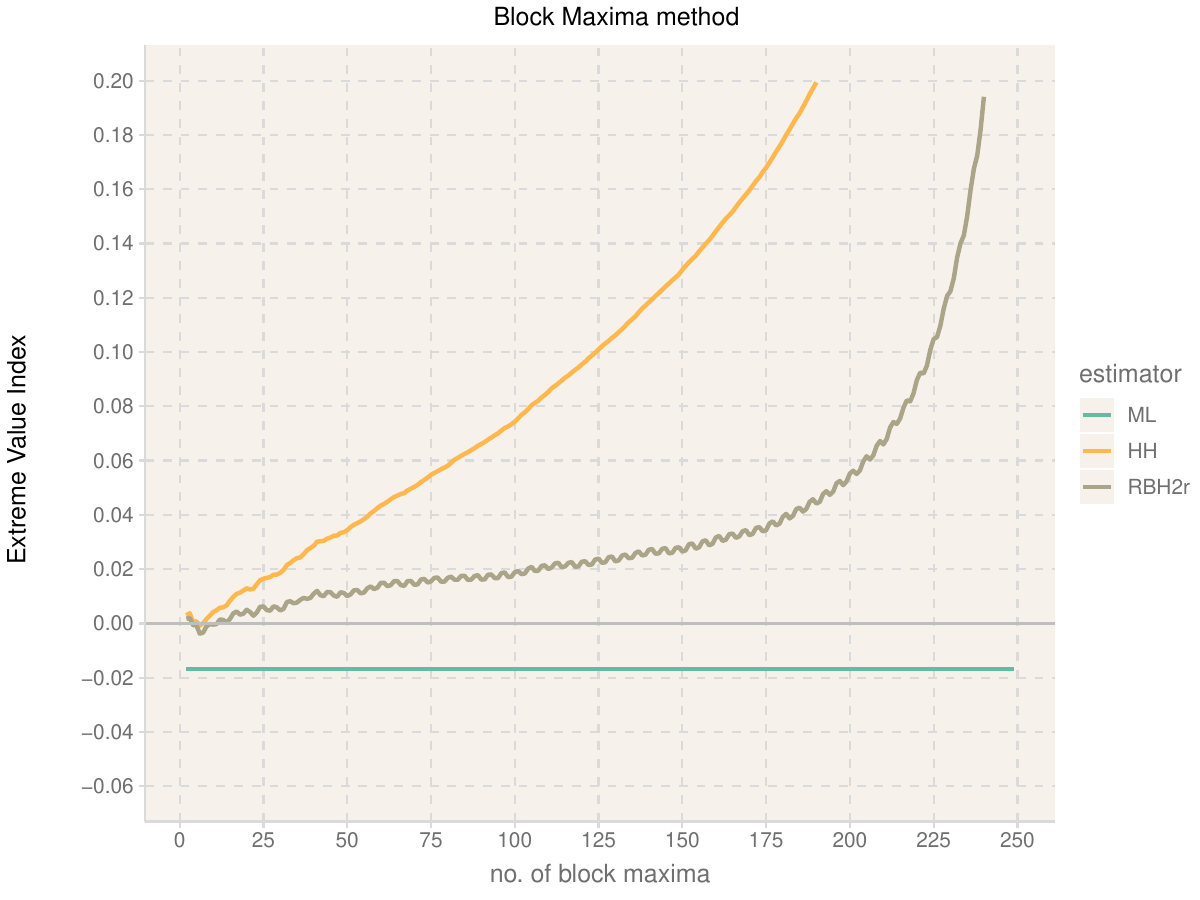} 
\includegraphics[scale=0.4]{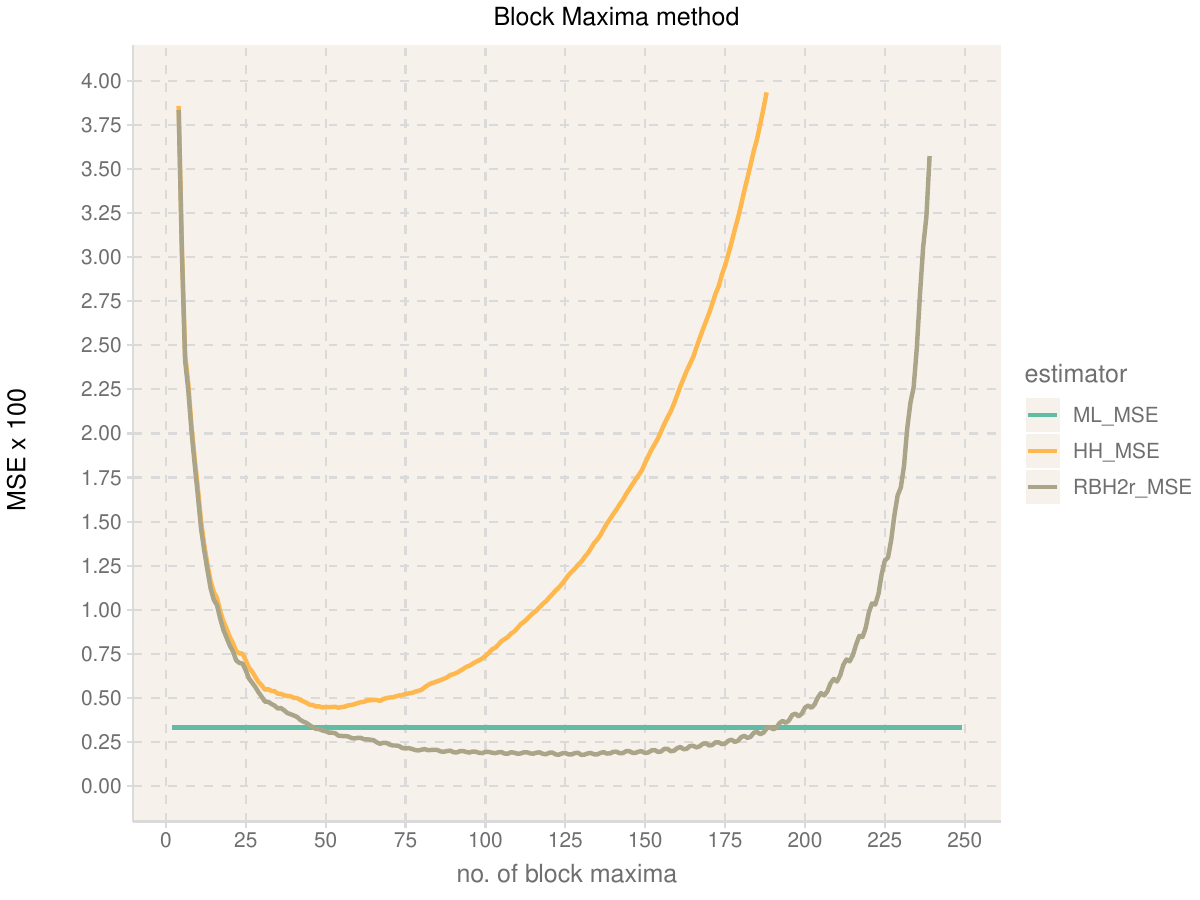}\\
\end{center}
\caption{Burr distribution satisfying \eqref{2RVlogVg} with $\gamma= 1/(\tau \lambda)= 0.75,\, 0.3$ and $\tilde{\rho} = -1/\lambda =-1,\, -2$, respectively top and bottom rows. Reduced-bias H2 (RBH2r) estimation is demonstrated accurate and surpasses the MLE for a wide range of intermediate values $k_0$. This is in spite of the MLE's smaller asymptotic variance for $\gamma \geq 1$ (top) and the misspecification of $\tilde{\rho}=-1$ to shorten H2 of its asymptotic bias (bottom).}
\label{Fig:RBBurr}
\end{figure}

\begin{figure}
\begin{center}
\includegraphics[scale=0.4]{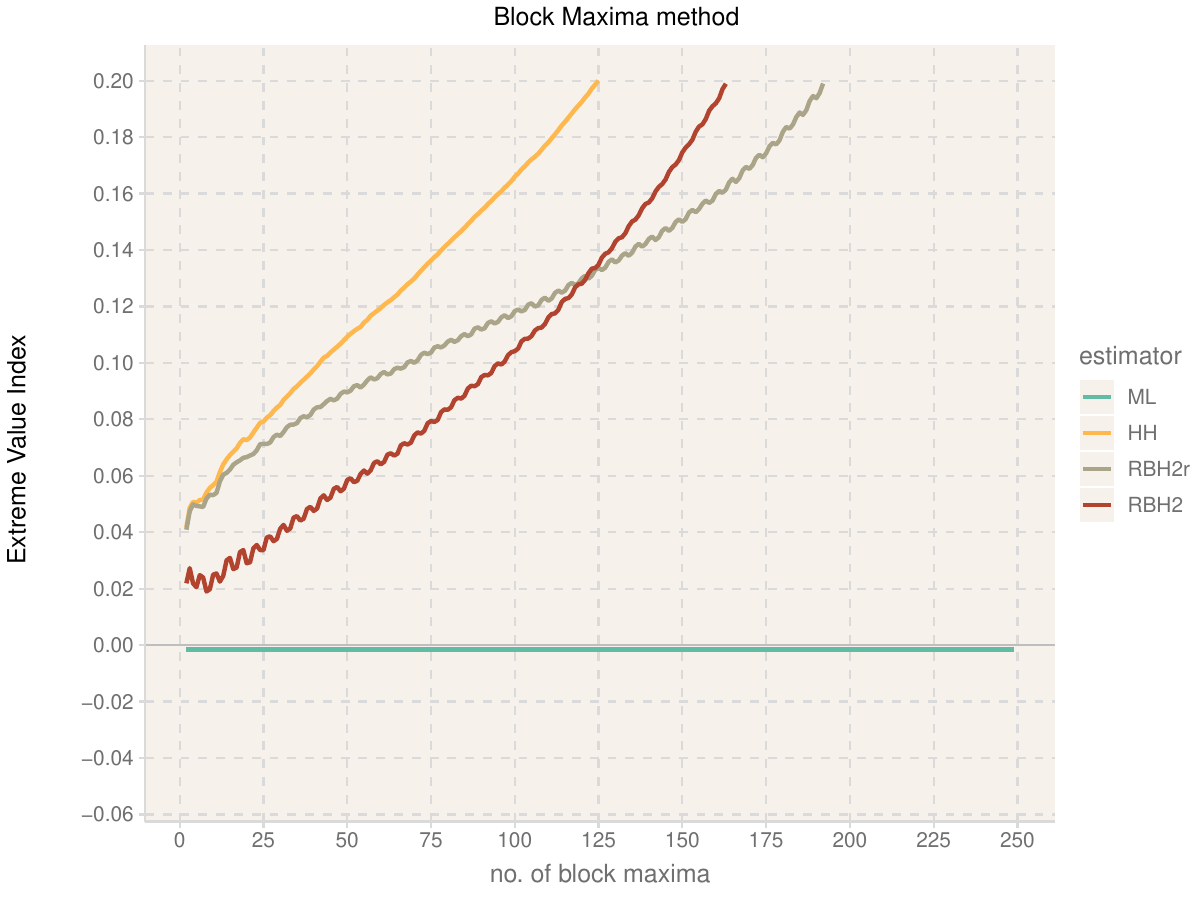} 
\includegraphics[scale=0.4]{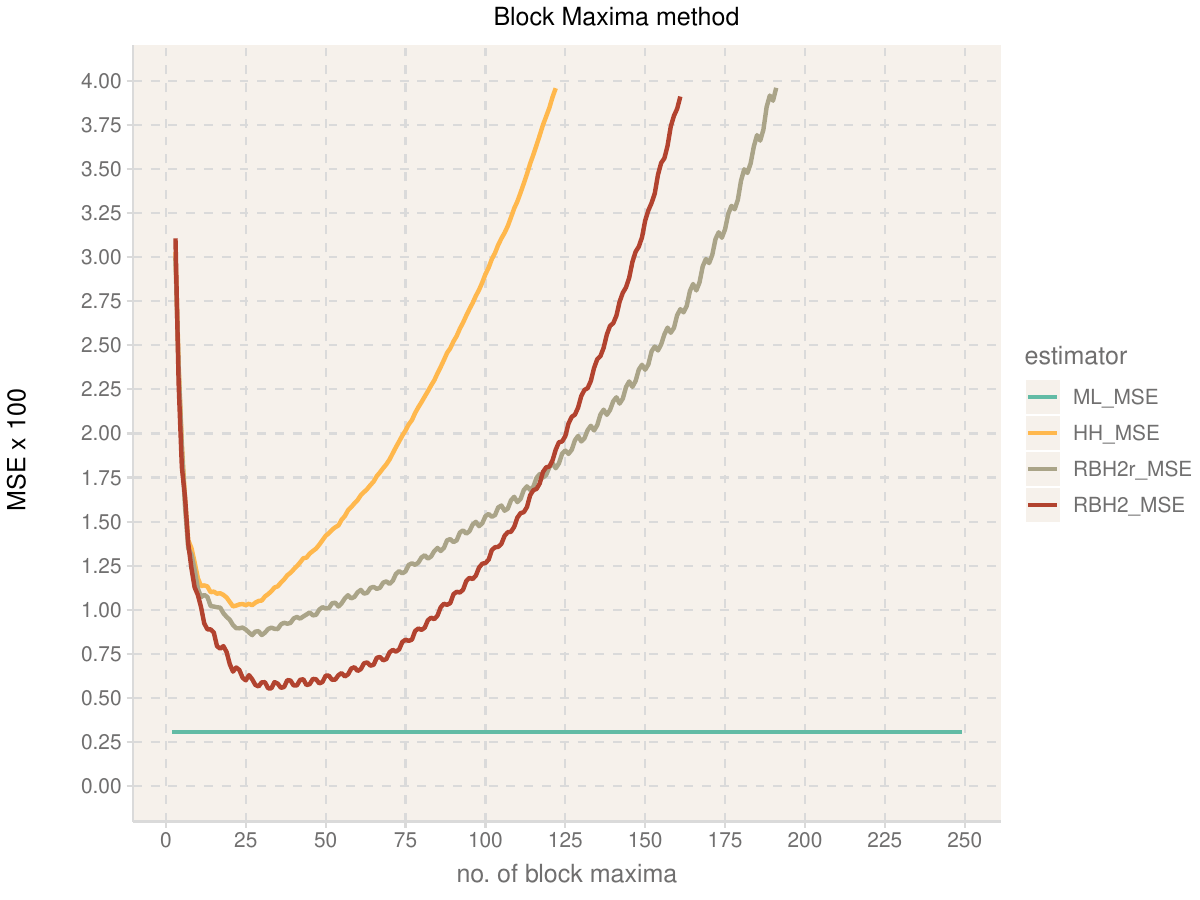}\\
\includegraphics[scale=0.4]{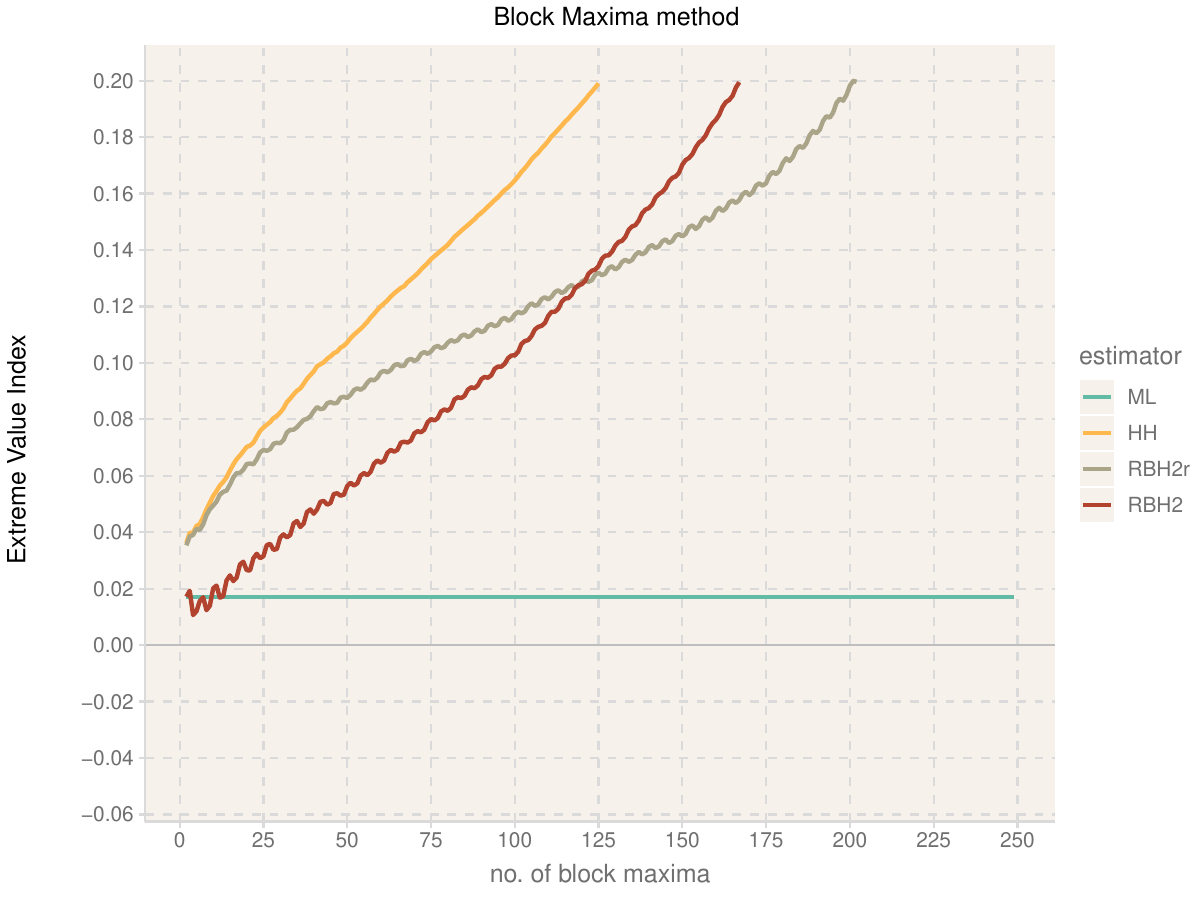} 
\includegraphics[scale=0.4]{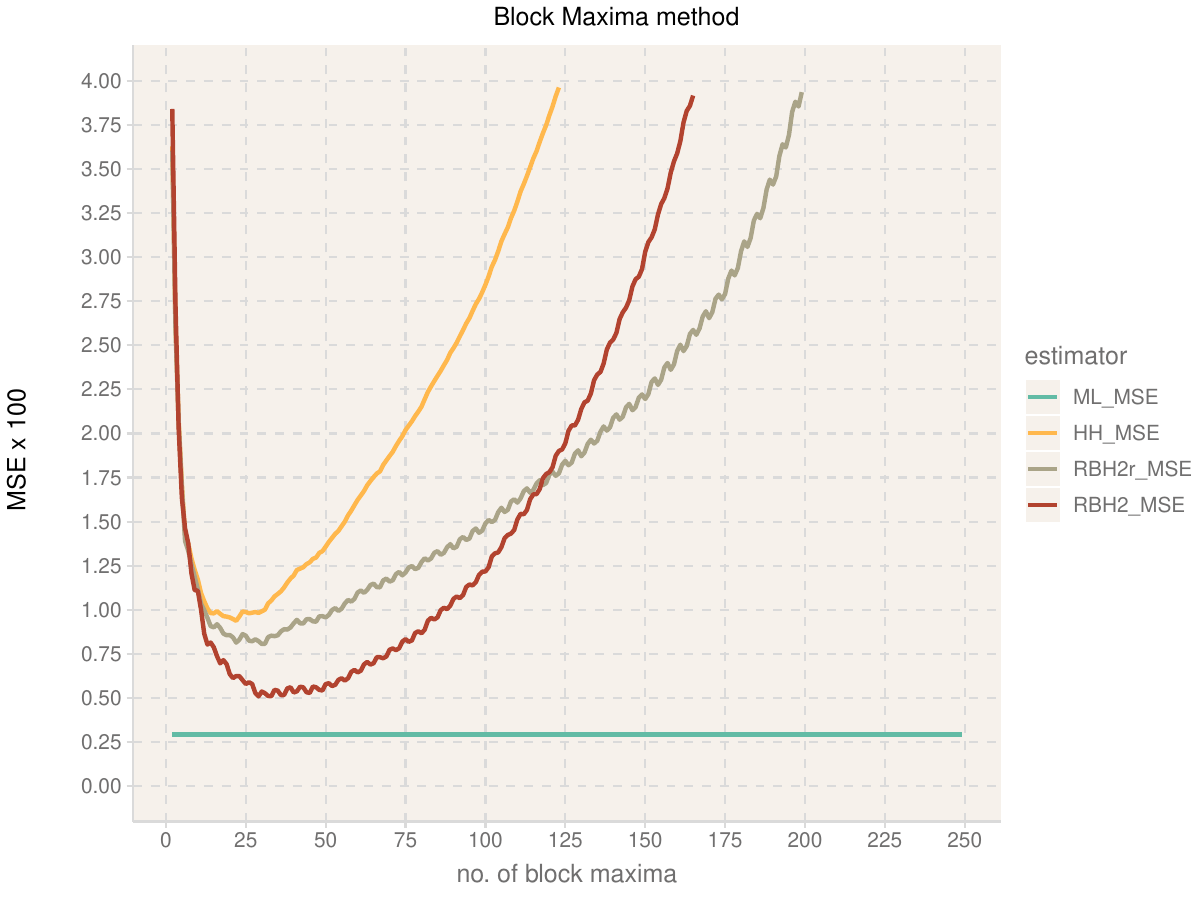}
\end{center}
\caption{Inference produced by H2 estimation based on the limiting GEV ascribed to BM and the stable GPD for POT exceedances, both of which here exemplified with $\gamma=0.25$ and $\tilde{\rho} = -\gamma$, is virtually the same.}
\label{Fig:RBGEVGPD025}
\end{figure}

Together, Figures \ref{Fig:RBFrechetPareto025} and \ref{Fig:RBGEVGPD025} compound a good graphic testament to the effectiveness of the proposed hybrid approach for unifying the, to date, polarised extreme value statistical streams: that grounded in the GEV-fit to BM data and those stemming from POT approaches predicated on the GPD-approximation to tail-related data. Additionally, Figure \ref{Fig:POTGPD} evidences that, even if the approximation bias is left unabated, the H2 estimator comes out in relatively equal standing to the conventional Hill estimator attached to $m=1$. When a small number of upper intermediate order statistics is employed, the H2-estimates' paths already bear a striking resemblance to those delivered by the foremost Hill estimator stemming from the POT approach. This is an example of how knowledge gained from one strand of development (POT approach) is poised to bolster new advances on the other (BM-based methodology) through the hybrid method.

\begin{figure}
\begin{center}
\flushleft \hspace{0.7cm}\includegraphics[scale=0.36]{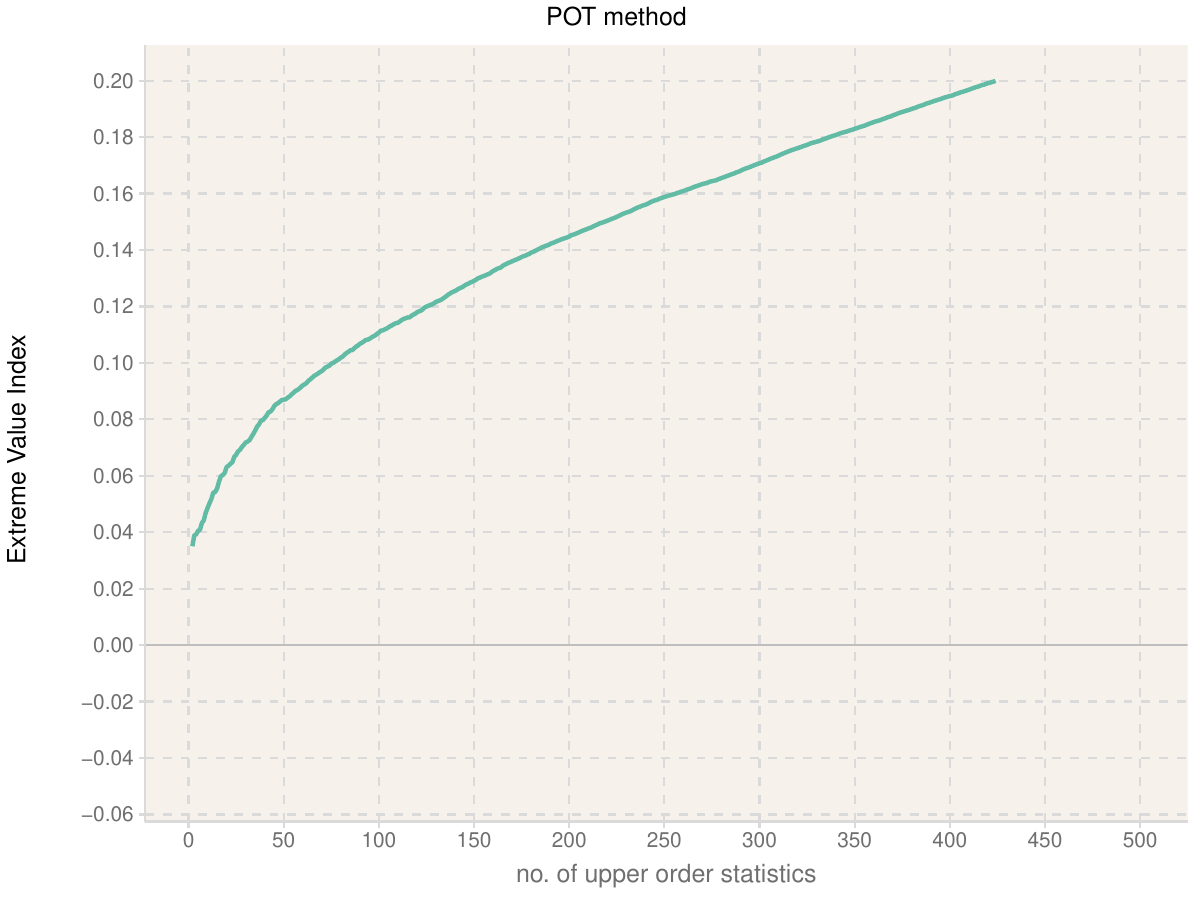}\hspace{0.5cm}
\includegraphics[scale=0.36]{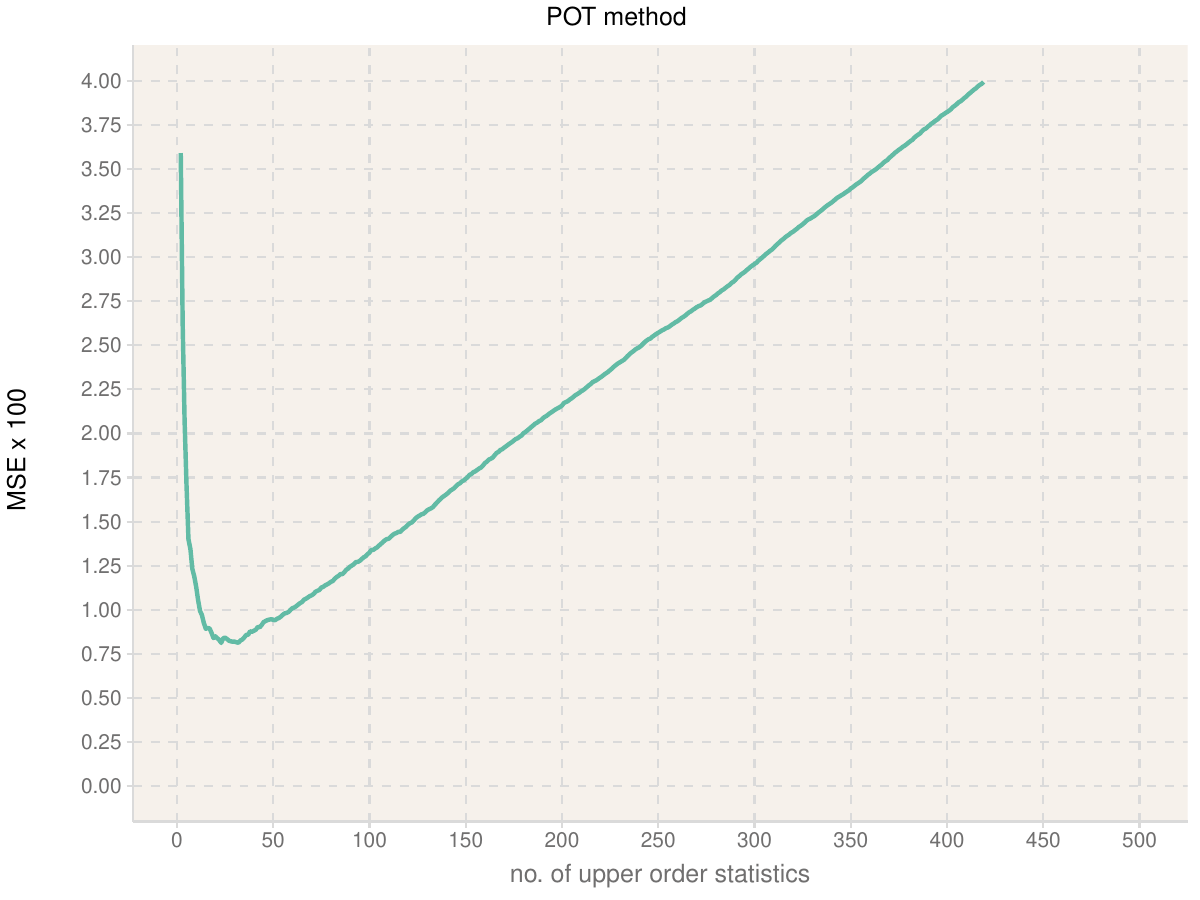}
\end{center}
\caption{Purely POT-GPD: conventional Hill  estimation  (i.e. H2 with $m=1$) for threshold exceedances on the basis of  samples of size $n=5000$ from the GDP with $\gamma=0.25$ and $\tilde{\rho}= -\gamma$.}
\label{Fig:POTGPD}
\end{figure}

%============== Beginning of Section Proofs ============
\section{Proofs} \label{Sec:Proofs}

%This section is devoted to the proof of the basic result in this paper.
This section is devoted to proving the foundational invariance results that bind together the stable distributions arising as the limit to each BM and POT approaches.

\begin{pfofLem}{~\ref{Lem:UnitFrechetOS}:}
For the proof of part (i), we find that the tail distribution function of the $1/2$-shifted unit Fr\'echet random variables
%$
 %   \overline{F}_{Z_i} (x-1/2) = 1 - \exp \{ - (x - 1/2)^{-1} \}, \, x\geq 1/2,$ $  i=1, 2, \ldots, k,
%$
develops as
\begin{equation}\label{eq:taylor}
	\overline{F}_{Z_i} (x-1/2) = 1 - \exp \{ - (x - 1/2)^{-1} \} = x^{-1} - \frac{x^{-3}}{12} \bigl( 1+ o(1)\bigr), 
\end{equation}
for large enough $x$. Indeed, the correction by $1/2$ harnesses the best possible approximation of the tail d.f. \eqref{eq:taylor} to that of the standard Pareto distribution \citep[cf.][]{Alzer1990}. Noticeably, the asymptotic representation \eqref{eq:taylor} is sharp, given the non-existent (equal to zero) second-order term. 
As a consequence, the crux of the proof is in invoking  Lemma 2.4.10 of \cite{deHaanFerreira2006}. 
With $Y_{1:k} \leq Y_{2:k} \leq \dots \leq Y_{k:k}$  the order statistics from a sample of i.i.d standard Pareto random variables with tail d.f.  $1 \wedge y^{-1}$, Lemma 2.4.10 of \cite{deHaanFerreira2006} applied locally, for $s=1$ and with  the intermediate sequence $\theta_k$ in place of $k$, determines that $\sqrt{k}\{(\theta_k/k)Y_{k - \theta_k:k} -1\} $ is asymptotically  standard normal. The precise statement (i) thus follows directly from \eqref{eq:taylor}.

\noindent Proof of part (ii):
Consider the $k$th ordered maxima of i.i.d. unit Fr\'echet  r.v.s and note that
\begin{equation*}
    M^{(m)}_{i:k} \id V\Bigl(m(Z_{i:k}+1/2)- \frac{m}{2} \Bigr).
\end{equation*}
Restricting focus to the $\theta_k +1$ larger order statistics $\bigl\{  M^{(m)}_{k-i:k}\bigr\}_{i=0}^{\theta_k}$, so that part (i) ensures each of associated  $\bigl\{  Z_{k-i:k}\bigr\}_{i=0}^{\theta_k}$ going to infinity with probability tending to one, expansion  \eqref{eq:taylor} yields
\begin{equation*}
   \bigl\{ M^{(m)}_{k-i:k}\bigr\}_{i=0}^{\theta_k} \mathop{\approx}^d \Bigl\{ V\Bigl(mY_{k-i:k}- \frac{m}{2} \Bigr) \Bigr\}_{i=0}^{\theta_k}.
\end{equation*}
 Finally, the  distributional relationship $Y \id \exp\{Z^{-1}\}$ delivers the stated result.
\end{pfofLem}

%------- begin of proof of lemma
\begin{pfofLem}{~\ref{Lem:BasicEmpProcess}:}
We borrow Theorem 2.1 of \cite{FerreiradeHaan2015}, albeit specialised in the i.i.d. unit Fr\'echet random variables $\{Z_k\}_{k\geq 1}$, from which we find for the basic unit Fr\'echet normalised quantile process that
    \begin{equation}\label{eq:Aux0}
        \sqrt{k}\Bigr(Z_{k-[ks]:k}-\frac{1}{-\log(1-s)}\Bigl)= -\frac{B_{k}\bigl(s\bigl)}{(1-s)\bigl(-\log(1-s)\bigl)^{2}}+o_{p}(1),
    \end{equation}
    for $s\in \bigl[\frac{1}{k+1},\frac{k}{1+k}\bigl]$, with $\{B_k(s)\}$ denoting a sequence of Brownian bridges on $[0,1]$.
    
\noindent Let $\bigl\{U^{*}_{n}\bigl\}_{n\geq 1}:=\bigl\{F(Z_{n})\bigl\}_{n\geq 1}$ be a sequence of i.i.d. standard uniform random variables stemming from $F(z) = \exp\{ z^{-1}\}$, $z\geq 0$. These are employed to define the quantile process $Q(s):=U^{*}_{k-[ks]:k}$ on $[0,1]$, where $U^{*}_{0:k } \equiv 0$. We now look at the inverse process of the initial (composite) quantile process
 \begin{equation*}
    H_{k}(s) = Z_{[k(1-s)]:k}:=-\frac{1}{\log \bigl( 1- U^{*}_{[ks]:k}\bigr)}= M(Q_{k}(s)), \quad s\in[0,1],
\end{equation*}
where we also define $M(s) := -1/\log (1-s)$ having continuous derivative $
    M'(s)=(1-s)^{-1}\bigl(-\log\bigl(1-s\bigl)\bigl)^{-2} <0$.
We are now ready to devise a similar approximation to \eqref{eq:Aux0} in terms of the  corresponding tail empirical process:
\begin{equation*}
  H_{k}(s)^{\leftarrow}= \bigl(M(Q_k(s)\bigl)^{\leftarrow}= \frac{1}{k}\displaystyle\sum_{i=1} ^{k}\one_{\{U_{i}^{*}\leq e^{-\frac{1}{s}}\}} \as \,\frac{1}{k}\displaystyle\sum_{i=1} ^{k}\one_{\{U_{i}\geq 1-e^{-\frac{1}{s}}\}},
\end{equation*}
with $\bigl\{U_{i}\bigl\}_{n=1}^{n}$ again independent random variables from the uniform distribution on the unit interval.
It results from \eqref{eq:Aux0}, through application of the non-increasing Vervaat's lemma \citep[cf.][Proposition 3]{Resnick2007}, that
\begin{equation}\label{eq:Aux1}
    \Bigl|\sqrt{k}\Bigl(\frac{1}{k}\displaystyle\sum_{i=1} ^{k}\one_{\{U_{i}\geq 1-e^{-\frac{1}{s}}\}}-(1-e^{-\frac{1}{s}}) \Bigl)-B_{k}(e^{-\frac{1}{s}})\Bigl|=o_{p}(1),
\end{equation}
uniformly for $s \geq \frac{1}{\log (k+1)} >0$, with the same sequence of Brownian bridges $\{B_k(s)\}$.

\noindent Next, we replace $s$ with $t-1/2$ everywhere in \eqref{eq:Aux1}, meaning that, for a sequence of reals $\vartheta_{k}> 0$ such that $\vartheta_k \log k \rightarrow \infty$, as $k \rightarrow \infty$, the appropriately normalised tail empirical process
\begin{equation}\label{eq:Alphak}
	\bigl\{\alpha_k(t) \bigr\}_{t > 1/2} := \sqrt{k}\Bigl\{\frac{1}{k}\displaystyle\sum_{i=1} ^{k}\one_{\{U_{i}\geq 1-\exp(-\frac{1}{t-1/2})\}}-\Bigl(1-\exp\bigl\{-\frac{1}{t-1/2}\bigr\}\Bigr) \Bigl\}_{t > 1/2}
\end{equation}
satisfies
\begin{equation*}
   \suprem{t \geq \vartheta_k/(\log (k+1))^{-1}} \, \Bigl|\alpha_k(t)-B_{k}(e^{-\frac{1}{t-1/2}})\Bigl|=o_{p}(1).
\end{equation*}
This enables to harness sharp bounds stemming from the Taylor's expansion for sufficiently large $t$:
\begin{equation}\label{eq:expansion}
    a(t) := 1- \exp\Bigl\{- \big(t -\frac{1}{2} \bigr)^{-1} \Bigr\}= t^{-1}\Bigl(1- \frac{t^{-2}}{12} + o(t^{-2} \bigl)\Bigr)
\end{equation}
%We observe that, with a sequence of i.i.d. standard Pareto random variables $\{Y_i\}_{i \geq 1}:= \bigl\{ (1-U_i)^{-1} \bigr\}$ with common d.f. $ 1- y^{-1}$, for $y \geq 1$, defined on the same probability space,
for obtaining
\begin{eqnarray*}
      & & \Bigl|  \frac{\alpha_k(t)}{\sqrt{k}} - \Bigl(\frac{1}{k}\displaystyle\sum_{i=1} ^{k}\one_{\{(1-U_i)^{-1}\geq s\}}-s^{-1} \Bigr)+\bigl(a(s)-s^{-1}\bigl)\Bigl|\\
   &\leq &  \Bigl|\frac{1}{k}\displaystyle\sum_{i=1} ^{k}\one_{\{U_{i}\geq a(s)\}}-a(s)\Bigl|+\Bigl|\frac{1}{k}\displaystyle\sum_{i=1} ^{k}\one_{\{(1-U_i)^{-1}\geq s\}}-\frac{1}{s}\Bigl|+o(1)=o_{p}(1),
\end{eqnarray*}
 as $k\rightarrow \infty$, due to the first two terms on the right h.s. of the inequality each resolving to a $O_{p}(1/\sqrt{k})$ by application of Donsker's theorem for the uniform (reduced) empirical process coupled with both the continuous mapping theorem and functional delta method. The third, deterministic term is asymptotically negligible given the slightly more stringent, albeit necessary requirement $\sqrt{k} \vartheta_k^{-3} =o(1)$ associated with \eqref{eq:expansion}. Hence, we have that
\begin{equation*}
   \sqrt{k} \Bigl|\frac{1}{k}\displaystyle\sum_{i=1} ^{k}\one_{\{U_{i}\geq 1-\exp(-\frac{1}{t-1/2})\}} -\frac{1}{k}\displaystyle\sum_{i=1} ^{k}\one_{\{(1-U_i)^{-1} \geq t\}}\Bigl| = O_p(1),
\end{equation*}
uniformly for every $t \geq \vartheta_k \log k$ and $t$ bounded away from $1/2$. It only remains to prove the desired a.s. convergence by way of a strengthening to the above convergence in probability. Indeed, for ascertaining
\begin{equation}\label{eq:asconvergence}
            \suprem{t \geq \, \vartheta_k\log k}\sqrt{k} \Bigl|\frac{1}{k}\displaystyle\sum_{i=1} ^{k}\one_{\{U_{i}\geq 1-\exp(-\frac{1}{t-1/2})\}} -\frac{1}{k}\displaystyle\sum_{i=1} ^{k}\one_{\{(1-U_i)^{-1} \geq t\}}\Bigl| \conv{a.s.} 0
\end{equation}
it suffices to show that, for every $\varepsilon>0$,
\begin{equation}\label{eq:borel}
    \displaystyle\sum_{k \geq 1} P\Bigl\{      \sqrt{k}     \Bigl|\frac{1}{k}\displaystyle\sum_{i=1}^{k}\one_{\{U_{i}\geq 1-\exp(-\frac{1}{t-1/2})\}}-\frac{1}{k}\displaystyle\sum_{i=1} ^{k}\one_{\{(1-U_i)^{-1} \geq t\}}\Bigl)\Bigl|>\varepsilon\Bigl\}\,<\infty,
\end{equation}
for all $ t\geq \vartheta_{k}\log k+ 1/2 =: \theta_k$, whereby we prove that the following tail probability becomes vanishingly small at a suitable (exponential) rate: for any $\delta>0$,
\begin{eqnarray*}
	& &  P \Bigl\{ \suprem{t \geq \theta_k} \sqrt{k}\, \Bigl| \Bigl( \frac{1}{k}\displaystyle\sum_{i=1} ^{k}\one_{\{U_{i}\geq 1-\exp(-\frac{1}{t-1/2})\}}- a(t) \Bigr) - \Bigl(\frac{1}{k}\displaystyle\sum_{i=1} ^{k}\one_{\{(1-U_i)^{-1} \geq t\}} - t^{-1} \Bigr) +  t^{-1} \bigl(1 + O(t^{-2}) \bigr)\Bigr| > \delta\Bigr\} \\
	&\leq&  P \Bigl\{ \suprem{t \geq \theta_k} \bigl| \alpha_k(t) \bigr| \geq \delta/3 \Bigr\} +  P \Bigl\{ \suprem{t \geq \theta_k}  \bigl|\Gamma_k(t) \bigr| \geq \delta/3\Bigr\}.
\end{eqnarray*}
To this end, we employ Theorem 1.1(3) from \cite{Mason2001} with $d=k \beta_k$,  for $\beta_k := 1-\exp\{ -(\vartheta_k \log k)^{-1}\}$ (a decreasing sequence in $k\geq 2$) in relation to the reduced empirical uniform empirical process arising from $\{\alpha_k(t)\}$ while setting $s = a(t)$ everywhere, thus rendering
\begin{equation*}
	\suprem{0< s \leq \beta_k} \bigl| \alpha^*_k(s) - B_k(1-s)\bigr| \geq k^{-1/2}(x - \log \vartheta_k)=: \varepsilon_k
\end{equation*}
for all $x \in \real$. Owing to the assumption that $\log \vartheta_k/\sqrt{k} = o(1)$, for every $\varepsilon >0$, there exists $k_0$ such that for $k \geq k_0$,
\begin{equation*}
	P \Bigl\{ \suprem{0< s \leq \beta_k} \bigl| \alpha^*_k(s) - B_k(1-s)\bigr| \geq \varepsilon \Bigr\}  \leq b e^{-c \varepsilon},
\end{equation*}
for some constants $b, c >0$. It is useful to note in this instance that $B_k(1-s) \id -B_k(s)$. Hence,
\begin{equation*}
	P \Bigl\{ \suprem{t \geq \theta_k} \bigl| \alpha_k(t) \bigr| \geq \delta/3 \Bigr\} \leq P \Bigl\{ \suprem{0< s \leq \beta_k} \bigl| \alpha^*_k(s) - B_k(1-s)\bigr| \geq \delta/3 \Bigr\} + P \Bigl\{ \suprem{0<s \leq \beta_k} \bigl| B_k(s)\bigr| \geq \delta/3 \Bigr\} \leq b e^{-c \delta } + e^{-2 \delta ^2 }.
\end{equation*}
The latter exponential term stems from Theorem 1.5.1 of \citet{CsorgoRevesz1981}. Both tail probabilities on the right hand-side decrease exponentially fast with decreasing $\beta_k$ (i.e. as $k\rightarrow \infty$).
 The convergence in \eqref{eq:asconvergence} thus holds by application of Borel-Cantelli lemma.

\noindent This reasoning applies seamlessly to the process $\{\Gamma_k(t)\}$ and its direct uniform counterpart $\{\Gamma^*_k(t)\}$ that originates from putting $t=1/s$. Indeed, the second Brownian bridge in the above follows from Donsker's theorem for the reduced empirical process, giving a distribution-free result outright. The proof is completed with the following note.
%The increments of a Brownian bridge are essentially normal random variables, with all their moments of any order (including fractional) being finite, whereby Lemma 2 in \citet[][Volume II, p. 239]{Feller1971}, ensures the validity of \eqref{eq:borel}.
The local modulus of continuity for Brownian bridge \citep[Theorem 1.9.1 in][]{CsorgoRevesz1981} establishes that, for $0<\delta<1/2$,
\begin{equation*}
           \lim_{h\downarrow 0}\suprem{0<x \leq h}\frac{\bigl|B(x+h)-B(x)\bigl|}{h^{1/2-\delta}} \mathop{=}^{a.s.} 0 \,.
\end{equation*}
In view of the expansion \eqref{eq:expansion}, it is fruitful to substitute $x=s^{-1}$ and use $h=O(s^{-3})$ so that
\begin{equation*}
            \lim_{h\downarrow 0}\suprem{0<s^{-1} \leq h}\frac{\bigl|B(s^{-1}+h)-B(s^{-1})\bigl|}{h^{1/2-\delta}}= \lim_{h\downarrow 0}\suprem{0<s^{-1} \leq h}\frac{\Bigl|B\bigl( \exp\{-1/(s-1/2)\}\bigr)-B(1-s^{-1})\Bigl|}{h^{1/2-\delta}}\mathop{=}^{a.s.} 0 \,.
\end{equation*}
Finally, application of Vervaat's lemma \citep[see][Proposition 3.3]{Resnick2007} coupled with L\'evy's distance theorem allows to revert to analogous formulation of \eqref{eq:asconvergence}, i.e., to that featuring the quantile process.
\end{pfofLem}
%--------------------

%======== APPENDIX =====================

\appendix

\section{Results concerning Condition B}\label{App:CondB1}

\begin{prop}\label{Prop:DoACondMax}
    For every $m \geq 1$, the distribution function $F^m$ is in the domain of attraction of an extreme value distribution $G_{\gamma}$ attached to a positive $\gamma$ (i.e. $F^m \in \mathcal{D}(G_\gamma)_{\gamma >0}$), in which case we write \begin{equation*}
        G_{\gamma}\bigl( \frac{x-1}{\gamma} \bigr) = \exp \{ - x^{-1/\gamma}\},
    \end{equation*}
    for all $x>0$, if and only if 
    \begin{equation*}
       \lim_{t\rightarrow \infty} \frac{1-F^m\bigl( x V(m(t-1/2)) \bigr) }{ 1 - F^m \bigl( V(m(t-1/2)) \bigr)} = x^{-1/\gamma},
    \end{equation*}
    for all $x>0$. Furthermore, if for each $s> 0$, $F_s$ is a probability distribution function satisfying the tail comparability condition
    \begin{equation*}
    	\limit{x} \, \frac{1-F^m_s(x)}{1-F^m(x)} = c(s) >0,
    \end{equation*}
    then, for all $x>0$, 
    \begin{equation*}
		\lim_{t\rightarrow \infty} \frac{1-F_s^m\bigl( x V(m(t-1/2)) \bigr) }{1 - F^m \bigl( V(m(t-1/2)) } = \Bigl(\frac{x}{c^\gamma(s)}\Bigr)^{-1/\gamma},
    \end{equation*}
    and $F^m_s \in \mathcal{D}(G_\gamma)$, $s>0$, with the same extreme value index $\gamma >0$.
\end{prop}    

\begin{proof}
 We begin with the precise definition of the quantile-type function $V_{\textrm{BM}}(t)$ pertaining to $F^m$:
 \begin{equation*}
 	V_{\textrm{BM}}(t):= \inf \bigl\{ x \in \real : \, \frac{1}{- m \log F(t)} \geq t \bigr\}, \quad t \geq 0,
\end{equation*}
from which expression it becomes clear that $V_{\textrm{BM}}(t) = V(mt)$. Since $F^m \in \mathcal{D}(G_\gamma)$ for some $\gamma >0$, then $F^m \in RV_{-1/\gamma}$, i.e.,
\begin{equation*}
	\limit{t} \frac{1-F^m(tx)}{1-F^m(t)} = x^{-1/\gamma},
\end{equation*}
while $\lim_{u \rightarrow \infty} V(u) = \infty$. Hence, by the left-continuity of $V_{\textrm{BM}}$, also a nondecreasing and positive function eventually, the first limit in the proposition holds locally uniformly for $x >0$. The indirect implication is routinely verified along the reverse  pathway. The second limit relation is verified through defining the distribution function $F_0(x)= \min \bigl(1, \, c(s) (1-F^m(x)\bigr)$ and using it into
\begin{equation*}
	\limit{t} \frac{1-F_s^m\bigl( x V(m(t-1/2)) \bigr) }{1 - F_0 \bigl(x  V(m(t-1/2)) } \frac{1-F_0\bigl( x V(m(t-1/2)) \bigr) }{1 - F_0 \bigl( V(m(t-1/2)) } = x^{-1/\gamma},
\end{equation*}
for all $ x>0$. The first part of the proposition is recovered under the convention $c(0) =1$, whereas the transitivity of the domain of attraction of $F_0$ to every $F^m_s$ is ascertained in \citet[][p.143]{TailTrend15}.
\end{proof}

\begin{lem}\label{Lem:IntRVofFm}
Given the conditions of Proposition \ref{Prop:DoACondMax}, it holds that
\begin{equation*}
    \lim_{t\rightarrow \infty} \int_{1}^{\infty}\frac{1-F^m\bigl( x V(m(t-1/2)) \bigr) }{ 1 - F^m \bigl( V(m(t-1/2)) \bigr)}\, \frac{dx}{x} = \gamma >0.
\end{equation*}
\end{lem}

\begin{proof}
	Heeding the associated uniform bounds stated in Condition B we have that, for every $\varepsilon,\delta>0$ there is $t_0= t_0(\varepsilon,\delta)$, such that for $t\geq t_0$ and $x\geq 1$
\begin{equation*}
    \Bigl|\frac{\overline{F^{m}}(xV(mg(t)))}{\overline{F^{m}}(V(mg(t)))}-x^{-\frac{1}{\gamma}}\Bigl| \leq \varepsilon x^{-\frac{1}{\gamma}+\delta},
\end{equation*}
we employ the triangle inequality followed by application of the dominated convergence to verify the approximation of the integrals in the theorem:
\begin{equation*}
    \biggl| \int_1^{\infty} \frac{\overline{F^{m}}\bigl( x V(m g(t)) \bigr) }{ \overline{F^{m}} \bigl( V(m g(t)) \bigr)}\, \frac{dx}{x} - \int_1^{\infty} x^{-1/\gamma} \, \frac{dx}{x} \biggr| \leq  \varepsilon \int_1^{\infty} \bigl| x^{-(\frac{1}{\gamma}-\delta)}\bigr|  \, \frac{dx}{x} < \infty,
\end{equation*}
for every $0<\delta< 1/\gamma $.
\end{proof}

\section{Second order regular variation development of Condition A2}\label{App:CondB2}
%-------- Lemma
\begin{lem}\label{Lem:2ndOrdLogShift}
Suppose that the function $V\in RV_{\gamma}$, $\gamma>0$, is such that the second order condition \eqref{eq:2ndRVofV} is satisfied with second order parameter $\rho \leq 0$. With $g(t):=t-1/2$, this implies for every $m\geq 1$,
\begin{equation*}
 \limit{t} \frac{\log V(mg(tx))-\log V(mg(t))-\gamma\log x}{\tilde{A}(mt)}= 
 \frac{x^{\tilde{\rho}}-1}{\tilde{\rho}},
\end{equation*}
for all $x>0$, where $\tilde{\rho} = \max(\rho, -1)$ and with
\begin{equation*}
 \tilde{A}(t)=\begin{cases}
A_{0}(t), & \rho > -1, \\
A_{0}(t) +\frac{m\gamma}{2}{t}^{-1}, &\rho=-1,\\
\frac{\gamma}{2} \bigl(\frac{t}{m}-\frac{1}{2} \bigr)^{-1}, & \rho < -1.
\end{cases} 
\end{equation*}
\end{lem}

%------------------------------
\begin{proof}
        For all $x>0$ and any $m \geq 1$, as $t \rightarrow \infty$,
    \begin{eqnarray*}
 & &\log V \bigl(mtx -\frac{m}{2} \bigr) - \log V\bigl(mt -\frac{m}{2} \bigr) - \gamma\log x\\
 &=& -\gamma\log x +\log V \Bigl( mt \bigl(x -\frac{1}{2t} \bigr) \Bigr) -\log V(mt) - \Bigl[\log  V \Bigl( mt \bigl(1 -\frac{1}{2t} \bigr) \Bigr) -\log V(mt)\Bigr]\\
 &=& -\gamma\log x + \gamma\log \Bigl(x-\frac{1}{2t}\Bigl)+A_{0}(mt)\frac{\bigl(x-\frac{1}{2t}\bigl)^{\rho}-1}{\rho}(1+o(1))-\Bigl[\gamma\log\Bigl(1-\frac{1}{2t}\Bigl)+A_{0}(mt)\frac{\bigl(1-\frac{1}{2t}\bigl)^{\rho}-1}{\rho}(1+o(1))\Bigr]\\
 &=&\gamma\log \Bigl(\frac{1-\frac{1}{2tx}}{1-\frac{1}{2t}} \Bigr)+A_{0}(mt)\bigl(1-\frac{1}{2t}\bigr)^{\rho} \frac{x^{\rho} \Bigl(\frac{1-\frac{1}{2tx}}{1-\frac{1}{2t}}\Bigl)^{\rho}-1}{\rho} + o \bigl( A_0(mt)\bigr)\,.
\end{eqnarray*}
We now define the building blocks that make up the latter, i.e.
\begin{equation}\label{RVblocks}
    \log V \bigl(mtx -\frac{m}{2} \bigr) - \log V\bigl(mt -\frac{m}{2} \bigr):= I+II + o \bigl( A_0(mt)\bigr).
\end{equation}
Application of the Taylor's expansion $\log(1+y)=y-\frac{y^2}{2}+o(y^2)$, with $y:=\bigl((1-\frac{1}{2tx})/(1-\frac{1}{2t})-1\bigl) \rightarrow 0$, yields:
\begin{equation*}
	I:= \gamma\log\Bigl(1+\bigl(\frac{1-\frac{1}{2tx}}{1-\frac{1}{2t}}-1\bigl)\Bigl) = \frac{\gamma}{2}(t-\frac{1}{2})^{-1}\frac{x^{-1}-1}{-1}-\frac{\gamma}{8}(t-\frac{1}{2})^{-2}\bigl(\frac{x^{-1}-1}{-1}\bigl)^2 + o(t^{-2})\,.
\end{equation*}
We find by analogous reasoning that the Taylor's expansion $(1+y)^a=1+ay+\frac{a(a-1)}{2}y^2+o(y^2)$ leads to
\begin{equation*}
 II:= A_0(mt)\bigl(1-\frac{1}{2t}\bigr)^{\rho} \frac{x^{\rho} \Bigl(\frac{1-\frac{1}{2tx}}{1-\frac{1}{2t}}\Bigl)^{\rho}-1}{\rho} =A_{0}(mt)\frac{x^{\rho}-1}{\rho}-\frac{A_{0}(mt)}{2t}\bigl(x^{\rho-1}-1\bigl)+o(\frac{A_{0}(mt)}{2t}).
\end{equation*}
Since $A_0$ is regularly varying with index $\rho\leq 0$ then, for any $m\geq1$, we have that
\begin{equation*}
    II=A_{0}(mt)\frac{x^{\rho}-1}{\rho}+o(A_{0}(mt))\,.
\end{equation*}
Finally, assembling $I$ and $II$ back into \eqref{RVblocks}, we obtain the full expansion:
\begin{equation*}
    \log V(m(g(tx))-\log V(mg(t))-\gamma \log x=\frac{\gamma}{2}\bigl(t-\frac{1}{2}\bigl)^{-1}\frac{x^{-1}-1}{-1}(1+o(1))+A_{0}(mt)\frac{x^{\rho}-1}{\rho}\bigl(1+o(1)\bigl)\,.  
\end{equation*}
\end{proof}	

%-----------------
\section{First and second order regular variation of $1-F^m$}\label{App:DoAFm}

\begin{prop}
Suppose there exists a function $\alpha$ of constant sign such that
\begin{equation}\label{eq:2RVtailF}
	\limit{t} \frac{ \frac{\overline{F}(tx)}{\overline{F}(t)} - x^{-1/\gamma}}{\alpha(t)} = x^{-1/\gamma} \frac{x^{\rho}-1}{\rho},
\end{equation}
for all $x>0$, where $\rho \leq 0$ is the second order parameter governing the rate of convergence. The convergence is necessarily locally uniform for $x>0$. Also $\alpha(t) \rightarrow 0$, as $t \rightarrow \infty$, and moreover is such that $|\alpha| \in RV_{\rho}$. It will be convenient as well as useful to rephrase the above second order condition in terms of the tail distribution function (d.f.) of the maximum of $m$ i.i.d. random variables with common d.f. $F$, namely \eqref{eq:2RVtailF} holds if and only if
\begin{equation*}
\limit{t} \frac{ \frac{1-F^m(tx)}{1-F^m(t)} - x^{-1/\gamma}}{\beta(t)} = x^{-1/\gamma} \frac{x^{\rho}-1}{\rho},
\end{equation*}
locally uniformly for $x>0$, with $\beta(t) \sim \alpha(t)$, as $t \rightarrow \infty$.
\end{prop}
\begin{proof}
The proof will be tackled by working essentially with the positive function
\begin{equation*}
	h(t) := t^{1/\gamma} \bigl( 1-F(t)\bigr) \in \textrm{ERV}(\rho), 
\end{equation*}
with auxiliary function $\alpha^{\star}(t):= t^{1/\gamma} \bigl( 1-F(t)\bigr) \alpha(t)$ that need not tend to zero. That is, for all $x>0$,
\begin{equation*}
	\limit{t}\frac{(tx)^{1/\gamma}\overline{F}(tx) - t^{1/\gamma}\overline{F}(t)}{\alpha^{\star}(t)} = \limit{t}\frac{h(tx) - h(t)}{\alpha^{\star}(t)} = \frac{x^{\rho}-1}{\rho}.
\end{equation*}
For devising the analogous first order extended regular variation \citep[cf.][App. B2]{deHaanFerreira2006} with intervening $F^m$, we note that $1-F^m(t)= \bigl(1-F(t)\bigr) \bigl( F^{m-1}(t) + F^{m-2}(t) + \ldots + F(t) + 1\bigr) \sim m(1-F(t))$, as $t \rightarrow \infty$, whose lower order term is non-concurrent with $B(t)$ so that
\begin{equation*}
	\frac{(tx)^{1/\gamma}(1-F^m(tx)) - t^{1/\gamma}(1-F^m(t))}{t^{1/\gamma}B(t)} \approx \frac{h(tx) - h(t)}{t^{1/\gamma}B(t)/m} \;\arrowf{t}\, \frac{x^\rho-1}{\rho}, 
\end{equation*}
by assumption, with auxiliary function given in terms of $\alpha(t)= B(t)/\bigl(m (1-F(t))\bigr) \rightarrow 0$. Therefore, setting
\begin{equation*}
	\beta(t) = \frac{m (1-F(t)) }{1-F^m(t)}\alpha(t) \sim \alpha(t) 
\end{equation*}
fulfils the second order condition on $1-F^m$, i.e.
\begin{equation*}
	\frac{(tx)^{1/\gamma}(1-F^m(tx)) - t^{1/\gamma}(1-F^m(t))}{t^{1/\gamma}(1-F^m(t)) \beta(t)} \arrowf{t}\, \frac{x^\rho-1}{\rho}, 
\end{equation*}
for all $x>0$, with $\beta(t) \rightarrow 0$ as $t\rightarrow \infty$.
\end{proof}

%========= Supplement ==========
\section*{Supporting information}
A comprehensive simulation study aimed at demonstrating the overall remarkable finite sample performance by the proposed hybrid estimation methodology is available with this paper at \url{http://bit.ly/491okdx}. Numerical experiments highlight the practical improvements that stand to be gained from the reduced-bias method as it furnishes a tuning-free, model-agnostic estimation procedure: with the H2 estimator, the selection of the block size $m$ is repurposed into the typical POT-selection of an intermediate number of larger data points (cf. Section~\ref{Sec:OptimalFraction}); the reduced-bias estimator further reveals fairly insensitive to a point-selection of $k_0$ as it allows to shift focus to a range of equally admissible values $k_0$, even when the second order parameter of the underlying model is misspecified at $\tilde{\rho}=-1$ (cf. Section~\ref{Sec:RBH2}). %The hybrid-Hill estimation stands  out as especially amenable to settings where dependence is present in the form of clusters since these entail effective reduction of the block size by the mean cluster size $1/\theta$ over threshold on the extremal index metrics (cf. Section~\ref{Sec:ennes}.

%========= References ==========
\bibliographystyle{apalike}
\bibliography{ReferHillFrechet}

\end{document}